\documentclass[12pt]{amsart}
\usepackage{amsmath,amssymb,amsthm,amscd}

\usepackage{graphicx}

\newtheorem{formula}{}[section]
\newtheorem{proposition}[formula]{Proposition}
\newtheorem{corollary}[formula]{Corollary}
\newtheorem{lemma}[formula]{Lemma}
\newtheorem{theorem}[formula]{Theorem}

\newtheorem{question}{Question}
\newtheorem{problem}{Problem}

\theoremstyle{definition}
\newtheorem{definition}[formula]{Definition}
\newtheorem{construction}[formula]{Construction}
\newtheorem{example}[formula]{Example}

\theoremstyle{remark}
\newtheorem{remark}[formula]{Remark}

\begin{document}

\title[Hyperbolic links associated to Hamiltonian 
subgraphs in simple $3$-polytopes]{Hyperbolic links associated to Hamiltonian 
subgraphs in simple $3$-polytopes}
%\,\footnote{\,{\it
%\uppercase{M}\uppercase{S}\uppercase{C}\,2000:\,} 55\uppercase{S}30,
%17\uppercase{B}56, 17\uppercase{B}70, 17\uppercase{B}10.
%\newline {\it \uppercase{K}eywords:\,}
%\uppercase{M}assey products, graded \uppercase{L}ie algebras,
%formal connection,
%\uppercase{M}aurer-\uppercase{C}artan equation,
%representation, cohomology.}}
%\author[D.V.~Gugnin]{Dmitry~Gugnin}
%\address{Department of Mechanics and Mathematics, Moscow State University, 119992 Moscow, Russia}
%\email{dmitry-gugnin@ya.ru}
\author[N.Yu.~Erokhovets]{Nikolai~Erokhovets}
\address{Department of Mechanics and Mathematics, Lomonosov Moscow State University}
\email{erochovetsn@hotmail.com}

\def\sgn{\mathrm{sgn}\,}
\def\bideg{\mathrm{bideg}\,}
\def\tdeg{\mathrm{tdeg}\,}
\def\sdeg{\mathrm{sdeg}\,}
\def\grad{\mathrm{grad}\,}
\def\ch{\mathrm{ch}\,}
\def\sh{\mathrm{sh}\,}
\def\th{\mathrm{th}\,}

\def\mod{\mathrm{mod}\,}
\def\In{\mathrm{In}\,}
\def\Im{\mathrm{Im}\,}
\def\Ker{\mathrm{Ker}\,}
\def\Hom{\mathrm{Hom}\,}
\def\Tor{\mathrm{Tor}\,}
\def\rk{\mathrm{rk}\,}
\def\codim{\mathrm{codim}\,}

\def\ko{{\mathbf k}}
\def\sk{\mathrm{sk}\,}
\def\RC{\mathrm{RC}\,}
\def\gr{\mathrm{gr}\,}

\def\R{{\mathbb R}}
\def\C{{\mathbb C}}
\def\Z{{\mathbb Z}}
\def\A{{\mathcal A}}
\def\B{{\mathcal B}}
\def\K{{\mathcal K}}
\def\M{{\mathcal M}}
\def\N{{\mathcal N}}
\def\E{{\mathcal E}}
\def\G{{\mathcal G}}
\def\D{{\mathcal D}}
\def\F{{\mathcal F}}
\def\L{{\mathcal L}}
\def\V{{\mathcal V}}
\def\H{{\mathcal H}}

%\newcommand{\rk}{\mathcal R_K}
%\DeclareMathOperator{\Dj}{\mbox{\textit{DJ}}}
%\thanks{The work was funded within the framework of the HSE University Basic Research Program}

%\address{Department of Mechanics and Mathematics, \\
%Moscow State University,\\
%Leninskie Gory, 119992 Moscow, Russia \vspace{1mm}\\
%E-mail: erochovetsn@hotmail.com}

%\thanks{The research was supported by the RFBR grant No 18-51-50005}



\begin{abstract}
We build a large family of hyperbolic links with an explicit 
decomposition of the complement into right-angled hyperbolic polytopes of finite volume.
Namely, in a series of papers A.D.Mednykn and A.Yu.Vesnin introduced a construction 
that for a given right-angled polytope $P$ in geometry $\mathbb L^3$, $\mathbb R^3$, $\mathbb S^3$,
$\mathbb L^2\times \mathbb R$, $\mathbb S^2\times \mathbb R$ and a Hamiltonian
cycle, theta-subgraph or $K_4$-subgraph $\Gamma$ in the $1$-skeleton of $P$ builds a geometric
$3$-manifold $N(P,\Gamma)$ with an involution $\tau$ such that $N(P,\Gamma)/\langle\tau\rangle\simeq S^3$.
The brach set of the corresponding $2$-sheeted branched covering $N(P,\Gamma)\to S^3$ is a link $C_\Gamma\subset S^3$
consisting of trivially embedded circles. This construction reformulated in the language of
toric topology works for such a subgraph $\Gamma$ in any simple $3$-polytope $P$ and gives a
topological $3$-manifold $N(P,\Gamma)$. We give a criterion when $S^3\setminus C_\Gamma$ 
has a complete hyperbolic structure of finite volume and generalize this criterion to
similar links in $3$-manifolds different from $S^3$. We prove that hyperbolic links $C_\Gamma\subset S^3$
are parametrized by nonselfcrossing Eulerian cycles, Eulerian theta-subgraphs and Eulerian $K_4$-subgraphs 
in hyperbolic right-angled $3$-polytopes of finite volume in $\mathbb L^3$ with $0$, $2$ or $4$ finite vertices.
The complement $S^3\setminus C_\Gamma$ is glued of $4$, $8$ or $16$ copies of the corresponding right-angled polytope.
We give a criterion when the link $C_\Gamma$ consists of mutually unlinked circles and 
prove that if such a link is nontrivial, then it contains the Borromean rings. The latter problem is
motivated by the Efimov effect in  quantum mechanics. We consider higher-dimensional analogues of
hyperbolic links $C_\Gamma$.
\end{abstract}
\maketitle
\tableofcontents
\setcounter{section}{0}
\section{Introduction}
The theory of knots and links is a classical area of mathematics developing since XIX century. 
One of the well-known directions in this area is the theory of hyperbolic links. These are links whose complements admit a complete hyperbolic structure of finite volume.  In the works \cite{M90, VM99S2} A.D.~Mednykh and A.Yu.~Vesnin
introduced a construction  that for a given right-angled polytope $P$ in geometry $\mathbb L^3$, $\mathbb R^3$, $\mathbb S^3$,
$\mathbb L^2\times \mathbb R$, $\mathbb S^2\times \mathbb R$ and a Hamiltonian
cycle, theta-subgraph or $K_4$-subgraph $\Gamma$ in the $1$-skeleton of $P$ builds a geometric
$3$-manifold $N(P,\Gamma)$ with an involution $\tau$ such that $N(P,\Gamma)/\langle\tau\rangle\simeq S^3$.
The brach set of the corresponding $2$-sheeted branched covering $N(P,\Gamma)\to S^3$ is a link $C_\Gamma\subset S^3$
consisting of trivially embedded circles. This construction reformulated in the language of
toric topology works for such a subgraph $\Gamma$ in any simple $3$-polytope $P$ and gives a
topological $3$-manifold $N(P,\Gamma)$. We give a criterion when $S^3\setminus C_\Gamma$ 
has a complete hyperbolic structure of finite volume (Theorem~\ref{thGHL}) and generalize this criterion to
similar links in $3$-manifolds different from $S^3$ (Theorem~\ref{thHL}). A a corollary we prove (Theorem~\ref{thHLP}) that 
hyperbolic links $C_\Gamma$ are parametrized by nonselfcrossing Eulerian cycles, 
theta-subgraphs and $K_4$-subgraphs 
in hyperbolic right-angled $3$-polytopes of finite volume in $\mathbb L^3$ with $0$, $2$ or $4$ finite vertices.  

In particular hyperbolic links $C_\Gamma$ corresponding to Hamiltonian cycles are parametrised by
nonselfcrossing Eulerian cycles in ideal right-angled $3$-polytopes $Q$ in $\mathbb L^3$. Usually ideal right-angled polytopes 
arise when the alternating diagram of a link is reduced to some canonical form (see, for example \cite{CKP22}). 
In our approach, the link $C_\Gamma$ consists of trivial circles
corresponding to vertices of the ideal right-angled polytope and their structure is defined by the Eulerian cycle. 
Any $k$-antiprism has a canonical Eulerian cycle. In this case, our decomposition of the complement
to the $(2k)$-link chain into $4$ antiprisms coincides with the decomposition described by W.P.~Thurston \cite[Example 6.8.7]{T02}.

The link corresponding to a Hamiltonian cycle in a simple $3$-polytope always contains the~Hopf link
consisting of two trivially embedded circles linked once. On the other hand, a theta-subgraph in the cube $I^3$
corresponds to the Borromean rings (Example~\ref{exthetaC}). We give a criterion 
when the link $C_\Gamma$ corresponding to a Hamiltonian theta-subgraph (Theorem~\ref{Th:unl}) or a 
Hamiltonian $K_4$-subgraph (Theorem~\ref{K4UL}) consists of mutually unlinked circles and prove that if
such a link is nontrivial, then it contains the Borromean rings. This is a partial answer to the question posed by Victor Buchstaber: {\it using technique of toric topology to build a rich family of Brunnian links, that is nontrivial links that become a set of trivial unlinked circles if any one component is removed}.  The question is motivated by the notion 
of a {\it Efimov state} \cite{E70} in  quantum mechanics. This is a bound state of three bosons such that the two-particle attraction is too weak to allow two bosons to form a pair. If one of the particles is removed, the remaining two fall apart. These three
bosons are governed by a three-body force (see \cite{EENNST25}). The Efimov sate is  symbolically depicted by the Borromean rings. Thus, our links $C_\Gamma$ may symbolically depict certain configurations of particles governed by three-body forces
with many triples of Efimov states. 

The Borromean rings is the first nontrivial  example of a~Brunnian link. Such links may correspond to Brunnian $n$-body
systems characterized by the complete absence of bound subsystems \cite{YFJ11}.
It turns out that except for the Borromean rings the link $C_\Gamma$ is not Brunnian. 
Recently in \cite{RV25}, a family of hyperbolic Brunnian links was constructed by other methods
starting from links $L_{3n+2}$ consisting of $3n+2$ components
with the complement $S^3\setminus L_{3n+2}$ decomposed into $4$ right-angled hyperbolic $(2n)$-antiprisms $A_{2n}$.

At the end of the article we consider a higher-dimensional analogue of the hyperbolic link $C_\Gamma$. We give a construction
which as a partial case produces a known Ivan\v{s}i\'{c} hyperbolic link consisting of $5$ tori in $S^4$ (see \cite{I04,I12, M24}). 
\section{Basic facts}
\subsection{The Steinitz theorem for $3$-polytopes}
For the theory of polytopes we refer to \cite{Z07,Gb03}.
By a {\it $3$-polytope} we mean a convex $3$-dimensional polytope. Moreover, in many cases 
we implicitly consider a {\it combinatorial polytope}, that is a class of combinatorial equivalence of $3$-polytopes.
 A $3$-polytope is {\it simple}, if any its vertex belongs to $3$ edges. A {\it face} of a plane graph is a connected component
 of its complement. By a graph $G(P)$ of a polytope $P$ we mean its $1$-skeleton consisting of vertices and edges. 
 \begin{theorem}[The Steinitz theorem]\label{StT}
A plane graph with more than one vertex is a graph of a $3$-polytope $P$ if and only if 
it has no loops and multiple edges, any its face is bounded by a simple edge-cycle, 
and if two such boundary cycles  intersect, then their intersection is a  vertex or and edge.
\end{theorem} 
Moreover, the Whitney theorem states that any two embeddings of the graph of a $3$-polytope to the plane
can be lifted to the combinatorial equivalence of the plane graphs (a bijection between sets of vertices, edges and faces
preserving the incidence relation).  See more details in \cite[Corollary 2.5.1]{BE17I}.

\subsection{Manifolds defined by vector-colorings of simple polytopes}
The following construction arises in toric topology \cite{BP15, DJ91}. We will give the construction and further
details following \cite{E24,E26}. The proofs of most fact mentioned below can be found there.

\begin{definition}\label{def:NP}
A vector-coloring of rank $r$ of a simple $n$-polytope $P$ is a mapping $\Lambda$ from the 
set of its facets $F_1,\dots, F_m$ to $\mathbb Z_2^r$, $F_i\to \Lambda_i$, such that 
$\langle\Lambda_1,\dots,\Lambda_m\rangle=\mathbb Z_2^r$. It corresponds to the space
$$
N(P,\Lambda)=P\times \mathbb Z_2^r/\sim, (p,a)\sim (q,b) \text{ if and only if } p=q\text{ and }a-b\in\langle \Lambda_i\colon p\in F_i\rangle. 
$$  
This space has an action of $\mathbb Z_2^r$. 

A vector-coloring $\Lambda$ is called {\it linearly independent}, if for any vertex 
$v=F_{i_1}\cap F_{i_2}\cap\dots\cap F_{i_n}$ the vectors $\Lambda_{i_1}$, $\Lambda_{i_2}$, $\dots$, $\Lambda_{i_n}$
are linearly independent.  
\end{definition}
It is known that for a linearly independent vector-coloring $\Lambda$ the space $N(P,\Lambda)$
is a closed manifold. More generally, $N(P,\Lambda)$ is a  topological manifold (possibly with a boundary)
if and only if for each vertex $v=F_{i_1}\cap\dots\cap F_{i_n}$ different nonzero vectors among $\Lambda_{i_1}$,
$\dots$,  $\Lambda_{i_n}$, are linearly independent. The boundary is glued of facets with $\Lambda_i=0$. The
manifold $N(P,\Lambda)$ is orientable if and only there is a linear function $c\in (\mathbb Z_2^r)^*$ such that
$c(\Lambda_i)=1$ for each nonzero $\Lambda_i$. Equivalently, in some coordinate system $\Lambda_i=(1,\lambda_i)$, 
$\lambda_i\in\mathbb Z_2^{r-1}$,  for each nonzero $\Lambda_i$. 

\begin{proposition}{\rm (see \cite[Corollary 1.12]{E24})}
Let $\Lambda$ be a vector-coloring of rank $r$ of a simple polytope $P$.
For a subgroup $H\subset \mathbb Z_2^r$ we have $N(P,\Lambda)/H\simeq N(P,\Lambda_H)$,
where $\Lambda_H$ is a vector coloring of rank $r-1$ obtained as the composition $\pi\circ\Lambda$, where
$\pi$ is the projection $\mathbb Z_2^r\to\mathbb Z_2^r/H$.
\end{proposition}

If $\Lambda_i=e_i$, where $e_1$, $\dots$, $e_m$ is a standard basis in $\mathbb Z_2^m$, then the space
$N(P,\Lambda)$ is called a {\it real moment-angle manifold} and is denoted $\mathbb R\mathcal{Z}_P$. It is an orientable manifold
with a canonical action of $\mathbb Z_2^m$.
Any space $N(P,\Lambda)$ is an orbit space of an action of a subgroup $H\subset \mathbb Z_2^m$
on $\mathbb R\mathcal{Z}_P$. The action of $H$ is free if and only if $\Lambda$ is linearly independent. In this case 
$\mathbb R\mathcal{Z}_P\to N(P,\Lambda)$ is a finite-sheeted covering.

\subsection{The Vesnin-Mednykh construction}
For a right-angled polytope $P\subset\mathbb L^n$ of finite volume define a vector-coloring
$\Lambda$ and the space $N(P,\Lambda)$ literally as in Definition~\ref{def:NP} 
(assuming that ideal vertices do not belong to $P$). We will call such a vector-coloring
{\it linearly independent}, if for any finite vertex $F_{i_1}\cap \dots \cap F_{i_n}$  of $P$ the vectors 
$\Lambda_{i_1}$, $\dots$, $\Lambda_{i_n}$ are linearly independent 
and for any edge $F_{i_1}\cap\dots\cap F_{i_{n-1}}$ containing the ideal vertex 
the vectors $\Lambda_{i_1}$, $\dots$, $\Lambda_{i_{n-1}}$ are linearly independent.

\begin{construction}\label{VMC}
The following construction goes back to the papers \cite{M85,MV86,V87, M90, VM99S2, V17} by A.Yu.~Vesnin and A.D.~Mednykh.
Let $P=P_1^{n_1}\times\dots\times P_k^{n_k}$ 
be either a product of compact Euclidean, spherical and hyperbolic right-angled polytopes $P_i\subset Y_i$,
$Y_i\in\{\mathbb R^{n_i},\mathbb S^{n_i},\mathbb L^{n_i}\}$
(we will call such a polytope right-angled) 
and $\mathbb X=Y_1\times\dots\times Y_k$ 
be the product of the corresponding spaces. Or  let $P$ be a right-angled polytope of finite volume in 
$\mathbb X=\mathbb L^n$.
Let $\mathcal{C}(P)$ be a right-angled Coxeter group generated by reflections $\rho_i$ in hyperplanes containing
the facets of $P$. This group is isomorphic to
$$
\langle \rho_1,\dots,\rho_m\rangle/(\rho_1^2=\dots=\rho_m^2=e, \rho_i\rho_j=\rho_j\rho_i,
\text{ if }F_i\cap F_j \text{ is an $(n-2)$-face}).
$$
The group $\mathcal{C}(P)$ acts on the ambient space discretely (that is, for any $x, y\in \mathbb X$ there is a neighbourhood
$U$ of $x$ such that $gy\in U$ only for a finite number of $g\in \mathcal{C}(P)$), and the polytope is a fundamental domain (that is, the polytopes $\{gP\}_{g\in \mathcal{C}(P)}$ fill all the space locally finitely and their interiors do not intersect). Also $P$ is the orbit space (this can be proved directly using the fact that $P$ is a fundamental domain), and for any point of $P$  (including ideal vertices of a polytope in $\mathbb L^3$) its stabilizer is generated by reflections in facets containing this point (see \cite[Theorem 1.2 in Chapter 5]{VS88} for
the case of $\mathbb X=\mathbb L^n$, $\mathbb S^n$ or $\mathbb R^n$).

Any vector-coloring  $\Lambda$ of $P$ defines an epimorphism $\varphi_\Lambda\colon\mathcal{C}(P)\to\mathbb Z_2^r$
by the rule $\rho_i\to \Lambda_i$. 
The subgroup $\mathcal{K}_{\Lambda}={\rm Ker}\,\varphi_\Lambda$ is a discrete group of isometries of $\mathbb X$. 
\begin{lemma}
The action of $\mathcal{K}_{\Lambda}$ on $\mathbb X$ is free if and only if $\Lambda$ is a linearly independent coloring.
\end{lemma}
\begin{proof}
The action is free if and only if for any $x\in \mathbb X$ we have $St_{\mathcal{K}_{\Lambda}}(x)=\{e\}$. Equivalently,
$\mathcal{K}_{\Lambda}\cap St_{\mathcal{C}(P)}(x)=\{e\}$. We have $x=g(y)$ for some $y\in P$. We know that
$St_{\mathcal{C}(P)}(y)=\langle\rho_j\colon y\in F_j\rangle$. Then 
$St_{\mathcal{C}(P)}(x)=\langle g\rho_jg^{-1}\colon y\in F_j\rangle$. We have 
${\rm Ker}\,\varphi_\Lambda\cap \langle g\rho_jg^{-1}\colon y\in F_j\rangle=
{\rm Ker}\,\varphi_\Lambda\cap \langle\rho_j\colon y\in F_j\rangle$. But for any finite point $y\in P$ the subgroup  
$\langle\rho_j\colon y\in F_j\rangle$ is isomorphic to $\mathbb Z_2^l$, where 
$l$ is the number of facets containing $y$, and the elements $\{\rho_j\colon y\in F_j\}$ 
form a basis of this subgroup as a vector space
over $\mathbb Z_2$. Then ${\rm Ker}\,\varphi_\Lambda\cap \langle\rho_j\colon y\in F_j\rangle=\{e\}$ if and inly if the images
of these elements are linearly independent. This holds for any $x$ if and only if  $\Lambda$ is linearly independent.
\end{proof}
\begin{corollary}
For linearly independent $\Lambda$ the space $\mathbb X/\mathcal{K}_{\Lambda}$ is a geometric manifold. 
\end{corollary}
\begin{remark}
In original papers by A.Yu.~Vesnin and A.D.~Mednykh the following argument is used: if $\Lambda$ is a linearly independent 
vector coloring of $P$, then ${\rm Ker}\,\varphi_\Lambda$ does not contain elements of finite order. From this it is deduced 
that  $\mathbb X/\mathcal{K}_{\Lambda}$ is a geometric manifold. As it is mentioned by the authors, this argument goes back to the paper \cite{A80}.
\end{remark}

%The subgroup ${\rm Ker}\,\varphi_\Lambda$ is a discrete group of isometries of $\mathbb X$
%and it acts freely on $\mathbb X$ with the orbit space being the geometric manifold. %It can be shown that 
%%this manifold is homeomorphic to $N(P,\Lambda)$ (see more details in \cite[Construction 1.2]{E26}). 
%This construction can be also applied  to a right-angled polytope of finite volume and a vector coloring $\Lambda$
%such that for any set of facets $\{F_i\}$ having a common point inside $\mathbb X$ the corresponding vectors $\Lambda_i$
%are linearly independent. For example, a right-angled polytope $P\subset\mathbb L^3$ may have vertices at infinity
%(ideal vertices). Each ideal vertex has valency $4$ and is contained in $4$ facets $(F_i, F_j, F_k, F_l)$,
%such that successive facets in this cyclic sequence have a common edge and should have different vectors $\Lambda(F)$,
%but the vectors $\Lambda_i$ and $\Lambda_k$ may coincide, as well as $\Lambda_j$ and~$\Lambda_l$.    
\end{construction}
%\begin{remark}\label{Rcusp}
%It is east to see that the manifold $\mathbb L^3/{\rm Ker}\,\varphi_\Lambda$ is homeomorphic
%to the interior of the manifold $N(\widehat P,\widehat{\Lambda})$, where $\widehat{P}$ is obtained from $P$ by cutting
%off all the ideal vertices and $\widehat{\Lambda}$ sends old facets $F_i$ to $\Lambda_i$ and new quadrangles to $0$.
%\end{remark}
Now we will describe a connection between the spaces $\mathbb X/\mathcal{K}_{\Lambda}$ and $N(P,\Lambda)$.

Let $\Lambda$ be any (not necessarily linearly independent) 
vector-coloring of rank $r$ of a compact right-angled polytope $P\subset\mathbb X$ or a right-angled polytope $P$ of finite volume
in $\mathbb L^n$ and $\varphi_\Lambda$ be the corresponding epimorphism $\mathcal{C}(P)\to\mathbb Z_2^r$. 
Since $\mathcal{K}_{\Lambda}$ is a normal subgroup in $\mathcal{C}(P)$, the group 
$\mathbb Z_2^r\simeq \mathcal{C}(P)/\mathcal{K}_{\Lambda}$ acts on $\mathbb X/\mathcal{K}_{\Lambda}$
by the rule $(a,[x])\to a[x]=[\varphi_\Lambda^{-1}(a)x]$, where $[x]$ is the equivalence class of $x$ under the equivalence
relation $x\sim y$
if and inly if $x=hy$ for some $h\in \mathcal{K}_{\Lambda}$. 

\begin{lemma}\label{lemLKNP}
There is an equivariant homeomorphism $N(P,\Lambda)\simeq\mathbb X/\mathcal{K}_{\Lambda}$.
\end{lemma}
\begin{proof}

Define a mapping $f\colon P\times\mathbb Z_2^r\to \mathbb X/\mathcal{K}_{\Lambda}$ by the 
rule $(x,a)\to a[x]$. It is continuous by construction. It is surjective, since $P$ is the fundamental domain of $\mathcal{C}(P)$.
If $f(x,a)=f(y,b)$, then $a[x]=b[y]$, and $[x]=(b-a)[y]$. Then $[x]=\left[\varphi_\Lambda^{-1}(b-a)y\right]$,
and $x=hgy$ for some $h\in \mathcal{K}_{\Lambda}$ and $g\in\mathcal{C}(P)$ such that
$\varphi_\Lambda(g)=b-a$. Moreover, $\varphi_\Lambda(hg)=b-a$.
In particular, $x=y$ (since $P$ is the orbit space of $\mathcal{C}(P)$), and $hg\in \langle \rho_j\colon x\in F_j\rangle$. 
Then $b-a\in\langle\Lambda_j\colon x\in F_j\rangle$. And visa versa, if $x=y$ and $b-a\in\langle\Lambda_j\colon x\in F_j\rangle$, then $f(a,x)=f(y,b)$. So, $f$ induces the continuous bijection 
$\widehat{f}\colon N(P,\Lambda)\to \mathbb X/\mathcal{K}_{\Lambda}$. It is equivariant, since
$\widehat{f}(a[(x,b)]=\widehat{f}([x,a+b])=(a+b)[x]=a(b[x])=a\widehat{f}([x,b])$. Now we need to prove only that
this mapping sends open sets to open sets. Indeed, for any point $x\in P$ consider its $\varepsilon$-neighbourhood $U_{\varepsilon}(x)$ in $\mathbb X$ such that $U_{\varepsilon}(x)\cap F_i\ne\varnothing$ if and only if $x\in F_i$.
For any point $[x,a]\in N(P,\Lambda)$ the neighbourhoods $U_{\varepsilon}(x)\cap P$ 
are glued in a neighbourhood $U_{\varepsilon}([x,a])\subset  N(P,\Lambda)$ of this point such that 
$U_{\varepsilon}([x,a])\cap U_{\varepsilon}([x,b])=\varnothing$ for  $[x,a]\ne [x,b]$. Such open sets form 
a base in the topology of $N(P,\Lambda)$. The preimage of $U_{\varepsilon}(x)\cap P$ under the projection map
$\mathbb X\to \mathbb X/\mathcal{C}(P)\simeq P$ is a disjoint union of  $U_{\varepsilon}(gx)$,
where $U_{\varepsilon}(gx)$ and $U_{\varepsilon}(hx)$ either do not interest (if $gx\ne hx$) or coincide (if $gx=hx$).
The image $\widehat{f}(U_{\varepsilon}([x,a]))$ in $\mathbb X/\mathcal{K}_{\Lambda}$ in lifted to a disjoint union of some 
$U_{\varepsilon}(gx)$ in $\mathbb X$. Hence, it is an open set and $\widehat{f}$ is a homeomorphism.
\end{proof}
\begin{remark}
This proof looks like the proof of the existence of a discrete group with a given fundamental polytope \cite[Theorem 1.8]{VS88},
which implies that $X$ is homeomorphic to the space 
\begin{gather*}
W(P)=P\times \mathcal{C}(P)/\sim,\\
\text{ where }(p,g_1)\sim(q,g_2)\text{ if and only if }p=q \text{ and }g_1^{-1}g_2\in
\langle \rho_i\colon p\in F_i\rangle.
 \end{gather*}
The construction of this space goes back to the papers \cite{V71, D83}.
\end{remark}
\begin{corollary}
For a linearly independent vector-coloring 
$\Lambda$ of rank $r$ of a compact right-angled polytope $P$ in $\mathbb X$ or a right-angled polytope $P$ of finite volume in 
$\mathbb X=\mathbb L^n$ the manifold $N(P,\Lambda)$ has geometric structure modelled on $\mathbb X$. 
\end{corollary}

If $P\subset \mathbb L^n$ has ideal vertices, then
consider a simple polytope $\widehat{P}$ obtained from $P$ by cutting off all its ideal vertices by different hyperplanes close to these vertices. Define a vector-coloring $\widehat{\Lambda}$ of $\widehat{P}$ by the rule: ``old'' facets $F_i$ are mapped to $\Lambda_i$ as before, and new cubical facets are mapped~to~$0$.
The space $N(\widehat P,\widehat{\Lambda})$ is a pseudomanifold with a boundary $\partial N(\widehat P,\widehat{\Lambda})$ 
glued of new cubical facets (see more details in \cite{E24}). Denote 
${\rm int}\,N(\widehat P,\widehat{\Lambda})\simeq N(P,\Lambda)\setminus \partial N(P,\Lambda)$.
\begin{lemma}\label{Rcusp}
Let $\Lambda$ be a vector-coloring of rank $r$ of a right-angled polytope $P\subset\mathbb L^n$
of finite volume. Then $N(P,\Lambda)\simeq {\rm int}\,N(\widehat P,\widehat{\Lambda})$.
\end{lemma}
\begin{proof}

If $P\subset \mathbb L^n$ has ideal vertices, then $P$ is homeomorphic to $\widehat{P}\setminus\{Q\}$,
where $Q$ is a disjoint union of cubical facets corresponding to cut vertices of $P$. The homeomorphism
can be constructed as follows. For each ideal vertex $v$ let $Q_v$ be the corresponding cubical facet.
Take a plane parallel to $Q_v$ and intersecting $\widehat{P}$ by a combinatorial cube $Q_v'$ close to $Q_v$.
Then the closed part of $\widehat{P}$ between $Q_v$ and $Q_v'$ is a polytope $R_v$ combinatorially equivalent to the $n$-cube
$I^n$. In $P$ we have a pyramid $P_v$ with apex $v$ and base $Q_v'$. Then a homeomorphism $P\to \widehat{P}$
sends $P\setminus\bigsqcup_v P_v$ to  $\widehat{P}\setminus\bigsqcup_v R_v$ by the identity map, and
each $P_v\setminus v$ to $R_v\setminus Q_v$ by the rule: the half-open interval $(v,x]$, $x\in Q_v'$, is mapped linearly
to the half-opened interval $(y,x]$, where $y=(v,x]\cap Q_v$.

This homeomorphism induces the homeomorphism $N(P,\Lambda)\to {\rm int}\,N(\widehat P,\widehat{\Lambda})$.

\end{proof}

\subsection{Hyperelliptic manifolds}
\begin{definition}\label{def:hypm}
A {\it hyperelliptic manifold}\index{hyperelliptic manifold} $M^n$ is an $n$-manifold with an~action of an involution $\tau$ such that $M^n/\langle\tau\rangle$ 
is~homeomorphic to $S^n$. The involution $\tau$ is called {\it hyperelliptic}\index{hyperelliptic involution}.  
\end{definition}

\begin{definition}
A graph $G$ is {\it cubic} graph if any its vertex has valency $3$. A {\it theta-subgraph} in a cubic graph
consists of two different vertices and three simple paths connecting these vertices. The paths have no common vertices except for their ends. A {\it $K_4$-subgraph} in a cubic graph consists of $4$ different 
vertices and $6$ simple paths connecting these vertices.
Each pair of vertices is connected by a path and different paths have no common vertices except for their ends.  

A {\it Hamiltonian cycle} in a graph $G$ is a cycle passing each vertex of $G$ exactly once.
A theta-subgraph or a $K_4$-subgraph $\Gamma$ in a graph $G$ is called {\it Hamiltonian}, if
any vertex of $G$ lies in $\Gamma$. 
For short we will call by a {\it Hamiltonian subgraph} a Hamiltonian cycle, a Hamiltonian theta-subgraph,
or a Hamiltonian $K_4$-subgraph.  

A {\it matching} $M$ of a graph $G$ is a disjoint set of its edges. A matching is {\it perfect} if it covers all the vertices of $G$. 
For a Hamiltonian subgraph $\Gamma\subset G$ denote by $M_\Gamma$ the matching in $G$
consisting of edges not lying in $\Gamma$.

A graph of a simple $3$-polytope $P$
is cubic.  By a subgraph or a matching in $P$ we mean a subgraph or a matching in its graph $G(P)$.

\end{definition}

\begin{construction}[A hyperelliptic manifold from a Hamiltonian subgraph]\label{HCC}
Let $\Gamma$ be a~Hamiltonian subgraph in a 
simple $3$-polytope~$P$.  The subgraph $\Gamma$ divides $\partial P\simeq S^2$ into $k=2$ (for the cycle), 
$k=3$ (for the theta-subgraph) or $k=4$ (for the $K_4$-subgraph) disks.  Each edge in $M_\Gamma$ divides one of the disks
into two disks. Thus, the adjacency graph of faces of $P$ lying in the closure of 
each component of $\partial P\setminus\Gamma$  is a tree and these faces can be colored in two colors (black and white) 
in such a way that adjacent faces have different colors. 
Let $a_1$, $\dots$, $a_k$, $\tau$ be a basis in $\mathbb Z_2^{k+1}$. Define $b_i=a_i+\tau$. 
Assign to each facet of $P$ in $i$-th component of $\partial P\setminus\Gamma$ the vector $a_i$ if it is white and $b_i$
if it is black. We obtain the vector-coloring $\widetilde{\Lambda}_{\Gamma}$ of rank $k+1$
and the orientable manifold $N(P,\widetilde{\Lambda}_{\Gamma})$ with the action of $\mathbb Z_2^{k+1}$. 
Then $\tau$ is a hyperelliptic involution and 
$N(P,\widetilde{\Lambda}_{\Gamma})/\langle\tau\rangle=N(P,\Lambda_{\Gamma})\simeq S^3$,
where $\Lambda_\Gamma$ is the composition $\pi\circ\widetilde{\Lambda}_{\Gamma}$, 
$\pi\colon \mathbb Z_2^{k+1}\to \mathbb Z_2^{k+1}/\langle\tau\rangle\simeq \mathbb Z_2^k$.
\end{construction}

The homeomorphism $N(P,\Lambda_{\Gamma})\simeq S^3$ can be seen as follows. 
All facets of $i$-th connected component of $P\setminus\Gamma$ are colored in the same vector $[a_i]\in \mathbb Z_2^k$,
and the vectors $[a_1],\dots,[a_k]$ form a basis in $\mathbb Z_2^k$. 

For the cycle $\Gamma$ there is a homeomorphism of $P\setminus x$, where $x$ is an interior point of some edge in $\Gamma$,
to the quaterspace defined by inequalities $y\geqslant 0$ and $z\geqslant 0$. 
Then the space $N(P,\Lambda_{\Gamma})\setminus [x\times \mathbb Z_2^2/\sim=pt]$ is equivariantly 
homeomorphic to $\mathbb R^3$ with the involutions $[a_1]$ and $[a_2]$ 
corresponding to the change of the sign of the $y$- and $z$-coordinates.

For the theta-subgraph $\Gamma$ let $v$ and $w$ be its vertices. Then there is a homeomorphism of $P\setminus w$
to the positive octant in $\mathbb R^3$ mapping $v$ to the origin and the paths to coordinate rays.  
Then the space $N(P,\Lambda_{\Gamma})\setminus [w\times \mathbb Z_2^3/\sim=pt]$ is equivariantly homeomorphic 
to $\mathbb R^3$ with the involution $[a_i]$ corresponding to the change of the sign of the $i$-th coordinate.

For the $K_4$-subgraph $\Gamma$ the complex in $\partial P$ given by edges and faces of this graph is 
homeomorphic to the boundary complex of the simplex $\Delta^3$, and $N(P,\Lambda_{\Gamma})$ is
equivariantly homeomorphic to the real moment-angle manifold $\mathbb R\mathcal{Z}_{\Delta^3}\simeq S^3$. 
It can be visualised similarly as for the theta-subgraph. Namely, there is a homeomorphism of $P$
to the simplex $\Delta^3$ that is the convex hull of the origin and the ends of the three basis vectors. Then the vectors
corresponding to three coordinate facets correspond to reflections in these facets.
Gluing $8$ copies of  $\Delta^3$ we obtain the octahedron $Oct^3$. Also for each octant the complement to the reflected
copy of $\Delta^3$ is homeomorphic to $\Delta^3\setminus\{\text{Origin}\}\simeq P\setminus\{v\}$, 
where $v$ is a vertex of $\Gamma$.
Then these complements are glued to $\mathbb R^3\setminus Oct^3$. 
%\begin{remark}\label{RUG}
%It can be shown that for a linearly independent vector-coloring $\Lambda$ of rank $r$ of a simple $3$-polytope 
%$P$ if there is a hyperelliptic involution in $\mathbb Z_2^r$, then there is a Hamiltonian cycle, theta- or $K_4$-subgraph
%$\Gamma$ in $P$ and a change of coordinates in $\mathbb Z_2^r$ such that $\Lambda=\widetilde{\Lambda}_\Gamma$
%and the involution is $\tau$ (see \cite[Theorem 11.5]{E24}).
% 
%\end{remark}

\subsection{Connection to the Vesnin-Mednykh approach to hyperelliptic manifolds}
If $\Gamma$ is a Hamiltonian cycle, theta- or $K_4$-subgraph in a compact right-angled polytope $P$ in the geometry $\mathbb X=\mathbb L^3$, $\mathbb R^3$, $\mathbb S^3$,
$\mathbb L^2\times \mathbb R$, or $\mathbb S^2\times \mathbb R$, then the manifold
$N(P,\widetilde{\Lambda}_{\Gamma})$ defined via Construction \ref{VMC} is exactly the hyperelliptic manifold
defined in \cite{M90, VM99S2} (actually, in \cite{M90} the manifold corresponding to Hamiltonian cycle in a right-angled polytope
was defined for all $5$ geometries, while in \cite{VM99S2} the authors consider only $X=\mathbb L^3$).   
Indeed, in \cite{M90} the manifold is defined implicitly using general facts on orbifolds, while in \cite[Section 3]{VM99S2} 
there is an explicit
construction of the hyperelliptic manifold using subgroups of the right-angled Coxeter group $\mathcal{C}(P)$ of 
a right-angled polytope $P$.

Namely, according to \cite[Theorem 3.1]{VM99S2} for a Hamiltonian cycle $\Gamma$ in the $1$-skeleton of $P$ the hyperelliptic manifold $M$ is defined
as $\mathbb L^3/{\rm Ker}\,\varphi$,  where the epimorphism 
$\varphi\colon \mathcal{C}(P)\to\mathbb Z_2^3$
corresponds to the coloring of facets of $P$ defined by $\Gamma$. The cycle $\Gamma$ divides $\partial P$
into two disks $U_1$ and $U_2$.  The facets lying in $U_1$ are colored in colors $\alpha=(1,0,0)$ and 
$\beta=(0,1,0)$, and the facets lying in $U_2$ are colored in $\gamma=(0,0,1)$ and $\delta=(1,1,1)$. This vector
coloring coincides with our $\widetilde{\Lambda}_\Gamma$ for $a_1=\alpha$, $b_1=\beta$, $a_2=\gamma$, and $b_2=\delta$.
We have $\varphi=\varphi_{\widetilde{\Lambda}_\Gamma}$, 
and $\mathbb L^3/{\rm Ker}\,\varphi\simeq N(P,\widetilde{\Lambda}_\Gamma)$ by Lemma~\ref{lemLKNP}.

Moreover, our sphere $N(P,\Lambda_\Gamma)=N(P,\widetilde\Lambda_\Gamma)/\langle\tau\rangle$ corresponds to the space 
$\mathbb L^3/\Gamma_1$ in \cite{VM99S2}, where 
$$
{\rm Ker}\,\varphi\subset \Gamma_1\subset \mathcal{C}^+_P\subset \mathcal{C}(P)
$$
are subgroups of index $2$.
Here  $\mathcal{C}^+(P)$ is the subgroup of orientation preserving elements and $\Gamma_1$ is defined 
as follows.  $\mathcal{C}^+(P)$  is generated by rotations through the angle $\pi$ around the edges of $P$. Each edge 
of $P$ can be  colored in $3$ colors $a=\alpha+\gamma=\beta+\delta$, $b=\alpha+\delta=\beta+\gamma$,
$c=\alpha+\beta=\gamma+\delta$, where the edge $F_i\cap F_j$ is colored in color $\widetilde\Lambda_i+\widetilde\Lambda_j$. 
We have $a+b+c=0$.
Then $\varphi$ induces the epimorphism $\psi\colon \mathcal{C}^+(P)\to \langle a,b,c\rangle\simeq
\mathbb Z_2^2$ sending the rotation in an edge to the color of this edge, and 
${\rm Ker}\,\varphi={\rm Ker}\,\psi$. By definition $\Gamma_1=\psi^{-1}(\langle c\rangle)$.

Indeed, in our construction $\tau=c$. It is easy to see that $\mathcal{C}^+(P)={\rm Ker}\,\varphi_{\Lambda_0}$ for
the vector-coloring $\Lambda_0$ or rank $1$ sending all facets of $P$ to $1$. For the vector-coloring 
$\Lambda_\Gamma=\pi_\tau\circ\widetilde{\Lambda}_\Gamma$, where  
$\pi_\tau\colon \mathbb Z_2^3\to \mathbb Z_2^3/\langle\tau\rangle\simeq\mathbb Z_2^2$ is the projection mapping, 
we have  $\varphi_{\Lambda_\Gamma}=\pi_\tau\circ \varphi$.

Also $\Lambda_0=\pi_{[a]}\circ \Lambda_{\Gamma}$, where 
$\pi_{[a]}$ is the projection 
$\mathbb Z_2^3/\langle\tau\rangle\to \mathbb Z_2^3/\langle\tau,a\rangle\simeq\mathbb Z_2$, and  
$\varphi_{\Lambda_0}=\pi_{[a]}\circ \varphi_{\Lambda_{\Gamma}}$.
In particular, 
$$
{\rm Ker}\,\varphi\subset {\rm Ker}\,\varphi_{\Lambda_{\Gamma}}\subset 
{\rm Ker}\,\varphi_{\Lambda_0}\subset \mathcal{C}(P).
$$
Then $\Gamma_1=\psi^{-1}(\langle\tau\rangle)=\varphi^{-1}(\langle\tau\rangle)$. Indeed, 
$\psi^{-1}(\langle\tau\rangle)\subset \varphi^{-1}(\langle\tau\rangle)$ by definition. On the other hand, if 
$\varphi(g)=h\in\langle\tau\rangle$, then $\varphi_{\Lambda_0}(g)=\pi_{[a]}\circ \pi_\tau\circ\varphi(g)=\pi_{[a]}\circ \pi_\tau(h)=0$ and $g\in \mathcal{C}^+(P)$.

Let us prove that $\varphi^{-1}(\langle\tau\rangle)={\rm Ker}\,\varphi_{\Lambda_{\Gamma}}$.
Indeed, if $g\in \varphi^{-1}(\langle\tau\rangle)$, then $\varphi(g)=h\in \langle\tau\rangle$,
and $\varphi_{\Lambda_{\Gamma}}(g)=\pi_\tau(h)=0$. If $0=\varphi_{\Lambda_\Gamma}(g)=\pi_\tau\circ \varphi(g)$, 
then $\varphi(g)\in\langle\tau\rangle$.

Thus, $\Gamma_1={\rm Ker}\,\varphi_{\Lambda_{\Gamma}}$ and  
$\mathbb L^3/\Gamma_1\simeq N(P,\Lambda_\Gamma)$ by Lemma \ref{lemLKNP}.
   
Similarly for a Hamiltonian theta-subgraph $\Gamma$  
the mapping $\varphi$ in \cite{VM99S2} is equal to $\varphi_{\widetilde{\Lambda}_\Gamma}$ for
$a_1=(1,0,0,0)$, $b_1=(0,1,1,1)$, $a_2=(0,1,0,0)$, $b_2=(1,0,1,1)$, $a_3=(0,0,1,0)$, $b_3=(1,1,0,1)$.
We have $\tau=(1,1,1,1)=a_i+b_i$, $i=1,2,3$.  Then the hyperelliptic manifold is 
$\mathbb L^3/{\rm Ker}\,\varphi\simeq N(P,\widetilde{\Lambda}_\Gamma)$ with the hyperelliptic involution $\tau$.
We have $S^3\simeq N(P,\widetilde{\Lambda}_\Gamma)/\langle\tau\rangle\simeq \mathbb L^3/\Gamma_1$,
where $\Gamma_1=\varphi^{-1}(\langle\tau\rangle)=\psi^{-1}(\langle\tau\rangle)$ 
(here $\psi$ is defined as above).

For a Hamiltonian $K_4$-subgraph $\Gamma$ in \cite{VM99S2} the construction is a little bit more complicated. The vector-coloring defining the epimorphism is not built explicitly. Instead first a projective-hyperelliptic manifold $\mathbb L^3/{\rm Ker}\,\varphi$ is constructed. (A three-dimensional manifold $M^3$ is called {\it projective-hyperelliptic} if there is an involution
$\tau$ such that the quotient space $M^3/\langle\tau\rangle$ is homeomorphic to the projective space $\mathbb RP^3$.)
Then the hyperelliptic manifold $\mathbb L^3/{\rm Ker}\,\widetilde{\varphi}$ is obtained as the two-fold orbifold
covering $\mathbb L^3/{\rm Ker}\,\widetilde{\varphi}\to \mathbb L^3/{\rm Ker}\,\varphi$
corresponding to the unbranched covering of the projective space by the
three-dimensional sphere.

In our terms the mapping $\varphi$ corresponds to the vector-coloring $\Lambda=\pi_\eta\circ\widetilde{\Lambda}_\Gamma$,
where $\pi_\eta\colon\mathbb Z_2^5\to\mathbb Z_2^5/\langle\eta\rangle\simeq\mathbb Z_2^4$ is the projection
map and $\eta=a_1+a_2+a_3+a_4+\tau$. If we take the basis $[a_1]$, $[a_2]$, $[a_3]$, $[a_4]$ in $\mathbb Z_2^5/\langle\eta\rangle$, then in this basis $[\tau]=(1,1,1,1)$,
$[b_1]=[a_1+\tau]=(0,1,1,1)$, $[b_2]=(1,0,1,1)$, $[b_3]=(1,1,0,1)$, and
$[b_4]=(1,1,1,0)$. This is exactly the vector-coloring defining $\varphi$ in \cite[Theorem 3.3]{VM99S2}. 
For the involution $[\tau]\in\mathbb Z_2^5/\langle\eta\rangle$ define 
$\Lambda'=\pi_{[\tau]}\circ\pi_\eta\circ\widetilde{\Lambda}_\Gamma$.  Then 
$$
N(P,\Lambda')\simeq N(\Delta^3,\widehat{\Lambda})\simeq\mathbb RP^3,
$$
where $\widehat{\Lambda}$ is a vector-coloring of rank $3$ of $\Delta^3$ sending its facets
to the vectors $(1,0,0)$, $(0,1,0)$, $(0,0,1)$ and $(1,1,1)$. It is a classical fact that  
$N(\Delta^3,\widehat{\Lambda})\simeq\mathbb RP^3$. By Lemma~\ref{lemLKNP}
$N(P,\Lambda')\simeq \mathbb L^3/{\rm Ker}\,\varphi_{\Lambda'}$.

In \cite{VM99S2} $\mathbb RP^3\simeq \mathbb L^3/{\rm Ker}\,\theta$. Here 
$\theta\colon \mathcal{C}(P)^+\to\mathbb Z_2^2$ is $\theta=\phi\circ\psi$, where
$$
\psi\colon \mathcal{C}(P)^+\to A=\{(x,y,z,w)\in \mathbb Z_2^4\colon x+y+z+w=0\}\simeq\mathbb Z_2^3
$$ 
is defined as before as $\left.\varphi\right|_{\mathcal{C}(P)^+}$, and $\phi\colon A\to\mathbb Z_2^3/\langle(1,1,1)\rangle\simeq \mathbb Z_2^2$ is defined as 
\begin{gather*}
\theta(1,1,0,0)=\theta(0,0,1,1)=[(1,0,0)],\quad \theta(1,0,1,0)=\theta(0,1,0,1)=[(0,1,0)],\\
\theta(0,1,0,1)=\theta(1,0,1,0)=[(0,0,1)],\quad \theta(1,1,1,1)=0.
\end{gather*}

We claim that ${\rm Ker}\,\varphi_{\Lambda'}={\rm Ker}\,\theta$. Indeed, $\varphi_{\Lambda'}=\pi_{[\tau]}\circ\varphi$.
Moreover, 
$$
\mathcal{C}(P)^+={\rm Ker}\,\pi_{[a_2+a_3]}\circ\pi_{[a_1+a_2]}\circ\varphi_{\Lambda'},
$$
therefore ${\rm Ker}\,\varphi_{\Lambda'}\subset \mathcal{C}(P)^+$. 
The mapping $\phi$ can be identified with the projection $\pi_{(1,1,1,1)}\colon A\to A/\langle (1,1,1,1)\rangle$, and thus with $\left.\pi_{[\tau]}\right|_A$.
Then 
$$
\theta=\phi\circ\psi=\left.\pi_{[\tau]}\right|_A\circ\left.\varphi\right|_{\mathcal{C}^+(P)}=
\left.(\pi_{[\tau]}\circ \varphi)\right|_{\mathcal{C}^+(P)}=\left.\varphi_{\Lambda'}\right|_{\mathcal{C}^+(P)}.
$$
In particular,  ${\rm Ker}\,\theta\subset{\rm Ker}\,\varphi_{\Lambda'}$. On the other hand,  
${\rm Ker}\,\varphi_{\Lambda'}\subset{\rm Ker}\,\theta$,
since ${\rm Ker}\,\varphi_{\Lambda'}\subset \mathcal{C}^+(P)$.

Thus, our projective space $\mathbb RP^3$ is the same as in \cite{VM99S2}.

At last, 
$\Lambda'=\pi_{[\eta]}\circ\pi_\tau\circ\widetilde{\Lambda}_\Gamma$, where $[\eta]\in\mathbb Z_2^5/\langle\tau\rangle$. 
Therefore,  we have the inclusions of subgroups 
$$
{\rm Ker}\,\varphi_{\widetilde{\Lambda}_\Gamma}\subset{\rm Ker}\,\varphi\subset{\rm Ker}\,\theta,\quad
{\rm Ker}\,\varphi_{\widetilde{\Lambda}_\Gamma}\subset{\rm Ker}\,\varphi_{\Lambda_\Gamma}\subset{\rm Ker}\,\theta,
$$
and a commutative diagram of $2$-fold branched coverings
$$
\begin{CD}
\mathbb L^3/{\rm Ker}\,\varphi_{\widetilde{\Lambda}_\Gamma}\simeq N(P,\widetilde{\Lambda}_\Gamma)
@>{/\langle\eta\rangle}>>N(P,\Lambda)\simeq \mathbb L^3/{\rm Ker}\,\varphi\\
  @V{/\langle \tau\rangle}VV @VV{/\langle[\tau]\rangle }V @.\\
S^3\simeq \mathbb L^3/{\rm Ker}\,\varphi_{\Lambda_\Gamma}\simeq N(P,\Lambda_\Gamma)@>{/\langle[\eta]\rangle }>>
N(P,\Lambda')\simeq \mathbb L^3/{\rm Ker}\,\theta\simeq\mathbb RP^3
\end{CD}
$$
where %the involutions $\tau$, $\eta$, $[\tau]$ and $[\eta]$ are isometries of the corresponding spaces and 
in the bottom line we have the unbranched covering $S^3\to\mathbb RP^3$. In
the top line we also have the unbranched covering, since  $\eta\notin\langle\widetilde{\Lambda}_{\Gamma}(F_i),\widetilde{\Lambda}_{\Gamma}(F_j), 
\widetilde{\Lambda}_{\Gamma}(F_k)\rangle$ for any vertex $F_i\cap F_j\cap F_k$ of $P$ and $\eta$ has no fixed points.
Thus, we can identify $\widetilde{\varphi}=\varphi_{\widetilde{\Lambda}_\Gamma}$, and for the Hamiltonian
$K_4$-subgraph $\Gamma$ our hyperelliptic manifold $N(P,\widetilde{\Lambda}_\Gamma)$  also
coincides with the hyperelliptic manifold $\mathbb L^3/{\rm Ker}\,\widetilde{\varphi}$
defined in \cite{VM99S2}.

\subsection{Classification of $VM$-hyperelliptic manifolds}
%Let is call by a {\it $VM$-hyperelliptic} manifold the manifold defined as follows. 
Let $P$ be a compact right-angled $n$-polytope in a $n$-dimensional geometry $\mathbb X=Y_1\times\dots\times Y_k$, 
 which is a product of  Euclidean, spherical, and hyperbolic geometries 
$Y_i\in\{\mathbb R^{n_i},\mathbb S^{n_i},\mathbb L^{n_i}\}$, $n_1+\dots+n_k=n$. 
For the commutator subgroup $\mathcal{C}'(P)$ of the right-angled Coxeter group $\mathcal{C}(P)$ we have 
$\mathcal{C}(P)/\mathcal{C}'(P)\simeq\mathbb Z_2^m$, where $m$ is the number of facets of $P$, and 
this factor group has a basis corresponding to the equivalence classes of reflections $\rho_i$
in facets of $P$. Then the subgroups $C\subset\mathcal{C}(P)$ such that $\mathcal{C}(P)/C$ is an abelian group
are exactly the subgroups containing $\mathcal{C}'(P)$ and exactly the subgroups defined as 
kernels of epimorphisms $C(P)\to\mathbb Z_2^r$. 
\begin{definition}
Let us call by a  {\it $VM$-hyperelliptic manifold} a manifold $M=\mathbb X/\widetilde{C}$, 
where $\mathcal{C}'(P)\subset \widetilde{C}\subset\mathcal{C}(P)$, the group  $\widetilde{C}$ acts freely on $\mathbb X$,
and there is a subgroup $C\subset \mathcal{C}(P)$ such that $\widetilde{C}\subset C$ has index $2$
and $\mathbb X/C$ is homeomorphic to~$S^n$.
\end{definition}
The subgroup $\widetilde{C}$ is a kernel of an epimorphism $\mathcal{C}(P)\to\mathbb Z_2^{\widetilde{r}}$, and 
$\mathbb Z_2^{\widetilde{r}}\simeq \mathcal{C}(P)/\widetilde{C}$ acts on $M$. The subgroup $C$ is also a kernel of an epimorphism 
$\mathcal{C}(P)\to\mathbb Z_2^{r}$. Since $\widetilde{C}\subset C$, there is an epimorphism 
$$
\mathbb Z_2^{\widetilde{r}}\simeq \mathcal{C}(P)/\widetilde{C}\to  \mathcal{C}(P)/C\simeq \mathbb Z_2^{r}
$$ 
with kernel $C/\widetilde{C}\simeq\mathbb Z_2$.  Then $\widetilde{r}=r+1$ and for a unique involution $\tau\in\mathbb Z_2^{r+1}$ 
generating this kernel we have $M/\langle\tau\rangle=\mathbb X/C\simeq S^n$. 
We call $\tau$ a {\it $VM$-hyperelliptic involution}.

Results of \cite{M90}  and \cite{VM99S2} can be reformulated in this language as follows: 
\begin{theorem}
\begin{enumerate}
\item Any Hamiltonian cycle in a compact right-angled $3$-polytope $P\subset\mathbb  S^3$, $\mathbb R^3$, $\mathbb L^3$, 
$\mathbb L^2\times \mathbb R$, and $\mathbb S^2\times\mathbb R$ induces  a $VM$-hyperelliptic manifold  \cite{M90};
\item Any Hamiltonian cycle, theta-subgraph or $K_4$-subgraph in a compact right-angled $3$-polytope 
$P\subset\mathbb L^3 $ induces a $VM$-hyperelliptic manifold \cite{VM99S2}.
\end{enumerate}
\end{theorem}

These results lead to natural problems.
\begin{problem}
To classify all $3$-dimensional $VM$-hyperelliptic manifolds.
\end{problem}
This problem was solved in \cite{E24}. Namely, \cite[Theorem 11.5]{E24} implies that in dimension $n=3$ any $VM$-hyperelliptic manifold $M$ is induced by a Hamiltonian cycle, theta-subgraph or $K_4$-subgraph. Moreover,
$VM$-hyperelliptic involutions are in bijection with such Hamiltonian subgraphs 
inducing $M$. Also the cases of more than one $VM$-hyperelliptic involutions were classified. In particular,
it was proved that a $VM$-hyperelliptic $3$-manifold $M$ corresponding to a Hamiltonian cycle has at most three 
$VM$-hyperelliptic involutions, and there are three of them if and only if $M$ is induced by three consistent Hamiltonian cycles,
where each edge of $P$ belongs to exactly two of them, and if and only if $M$ is a rational homology sphere.

Let us remark that in dimension $n=2$ the classification of $VM$-hyperelliptic manifolds is straightforward.

The following problem was stated by A.D.~Mednykh during discussions of results from \cite{E24}.
\begin{problem}
To build and classify $VM$-hyperelliptic manifolds in dimensions $n\geqslant 4$.
\end{problem}
This problem in dimension $n=4$ was solved in \cite{E26}. 
It was proved that any $VM$-hyperelliptic manifold $M$ is induced by a so-called Hamiltonian $C(4,r)$-subcomplex in 
a right-angled polytope $P$ and $VM$-hyperelliptic involutions are in bijection with such subcomplexes inducing $M$. In particular,
it was proved that $4$-dimensional $VM$-hyperelliptic manifolds exist only in geometries 
$\mathbb X=\mathbb S^4$, $\mathbb S^3\times \mathbb R$,  $\mathbb S^2\times \mathbb S^2$, 
$\mathbb S^2\times \mathbb R^2$, $\mathbb S^2\times\mathbb L^2$, and $\mathbb L^2\times\mathbb L^2$,
where in the last geometry the example is given by a $2$-fold branched covering corresponding to the 
Ivan\v{s}i\'{c} link consisting of $5$ tori in $S^4$ \cite{I04,I12}. This link  
is a $4$-dimensional generalisation of the Borromean rings: 
each two tori form a trivial link, but each three do not \cite[Theorem 1]{M24}. In geometries $\mathbb X=\mathbb R^4$,
$\mathbb L^4$, $\mathbb L^3\times \mathbb R$, and $\mathbb L^2\times\mathbb R^2$ there are no $VM$-hyperelliptic
manifolds.

For $n>4$ in \cite{E26} it was proved that any Hamiltonian $C(n,k)$-subcomplex in a right-angled polytope 
induces a $VM$-hyperelliptic manifold,
but the converse statement remains an open question.

If we do not take into account geometry, then the following
notion of a {\it generalized  $VM$-hyperelliptic manifold} arises. 
\begin{definition}\label{def:genvm}
A  {\it generalized hyperelliptic $VM$-manifold} $M$ is a manifold $N(P,\Lambda)$
corresponding to a linearly independent vector-coloring of rank $r$ of a simple $n$-polytope such that the group $\mathbb Z_2^r$
canonically acting on $N(P,\Lambda)$ contains a hyperelliptic involution. This involution we call
a~{\it generalized $VM$-hyperelliptic involution}.
\end{definition} 
Any $VM$-hyperelliptic manifold is a generalized $VM$-hyperelliptic manifold. 
All the results of \cite{E24,E26} are formulated and proved for generalized hyperelliptic $VM$-manifolds.
A particularly interesting family of generalized $VM$-hyperelliptic manifolds is formed
by {\it generalized $VM$-hyperelliptic small covers} corresponding to linearly independent vector-colorings of rank $n$.
It can be proved that in dimension $n=4$ there are no $VM$-hyperelliptic small covers. The first example
of a $4$-dimensional {\it generalized}  $VM$-hyperelliptic small cover was found by A.O.~Koretskii in \cite{K24,K24B}.

\subsection{Links defined by Hamiltonian subgraphs}

The mapping $N(P,\widetilde{\Lambda}_{\Gamma})\to N(P,\Lambda_{\Gamma})\simeq S^3$ is a $2$-sheeted branched covering with the following branch set (see details in \cite[Section 4.5]{E26}).
The edges of $P$ not lying in $\Gamma$ form a matching $M_{\Gamma}$ of  $G(P)$ and the preimage of this set in 
$N(P,\widetilde{\Lambda}_{\Gamma})$ and in $S^3$ is a disjoint set of circles $C_{\Gamma}$. This link is the branch set of the covering. 
\begin{itemize}
\item For $k=2$ each edge of $M_{\Gamma}$ corresponds to a circle glued of two copies of the edge. 
\item For $k=3$ each edge of $M_{\Gamma}$ corresponds either to a~circle glued of $4$ copies of this edge 
(if the edge has vertices on different paths of $\Gamma$), 
or to a~pair of circles each glued of $2$ copies of the edge (if the edge has vertices on the same path of $\Gamma$). 
\item For $k=4$ it corresponds either to a~pair of circles glued of $4$ copies of the edge (if the edge has vertices on different paths of $\Gamma$), or to $4$ circles each glued of $2$ copies of the edge (if the edge has vertices on the same path of $\Gamma$). 
\end{itemize}

%добавить теорему про единственность конструкции.
The link $C_\Gamma$ can be visualised in $\mathbb R_3$ as follows. For the above identification
of $P\setminus\{point\}$ with a quaterspace ($k=2$), orthant ($k=3$) and $\Delta^3\setminus\{point\}$,
we can draw each edge in $M_\Gamma$ as a semicircle lying in the part of a plane corresponding
to the facet of $\Gamma$, if its vertices  lie on the same path (or circle) of $\Gamma$, and by a straight
segment, if they lie on different paths. Then after all the ``reflections'' producing  $\mathbb R^3$ from $P$
we obtain the explicit realization of  $C_\Gamma$ in $\mathbb R^3$. The author is grateful to D.A.~Tsygankov~\cite{T25}
and D.V.~Chepakova~\cite{C23} for the idea of this realization for $k=2$.
\begin{remark}
The detailed description of the link $C_\Gamma$ corresponding to a Hamiltonian cycle $\Gamma$ in terms of bipartite chord
diagrams is given by Vladimir Gorchakov in \cite{G24}.  He proved that the bridge index of $C_\Gamma$ is equal to $l$, where 
$2l$ is the number of vertices of $P$.
\end{remark}
\begin{example}
In Fig.~\ref{Hopf} we show the link corresponding to a Hamiltonian cycle in the simplex. It is the Hopf link
consisting of two trivially embedded circles liked once.
\end{example}
\begin{figure}
\begin{center}
\includegraphics[width=0.5\textwidth]{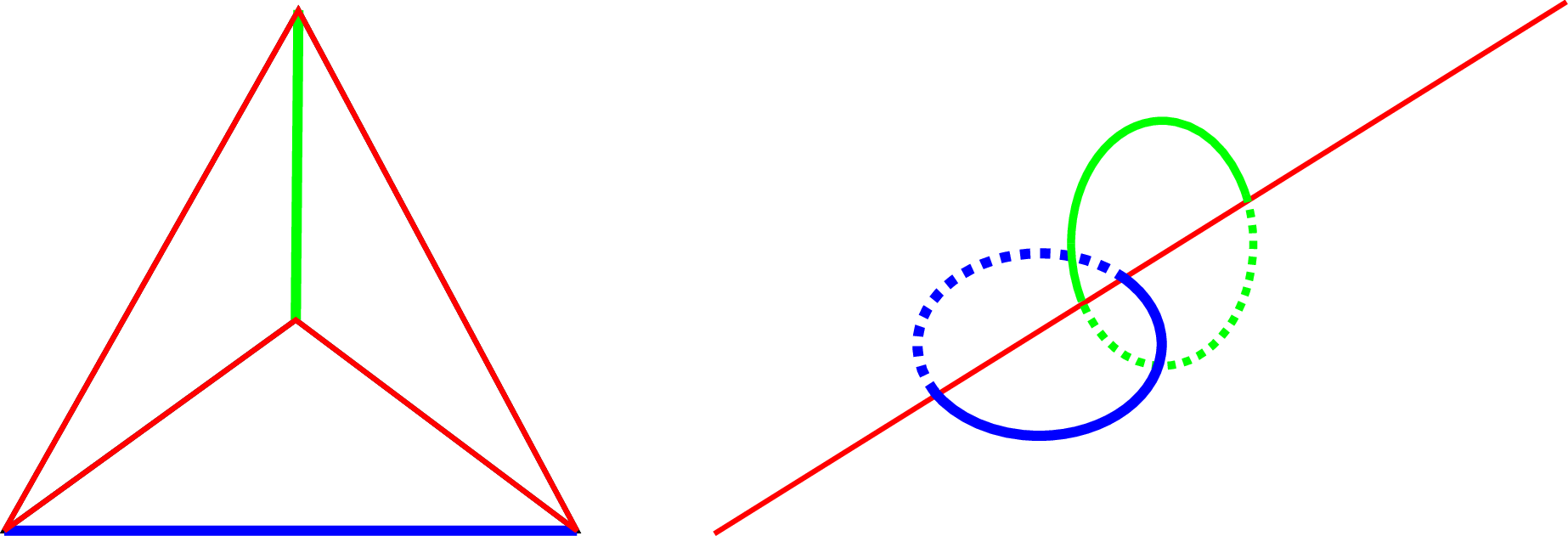}
\caption{The Hopf link corresponding to a Hamiltonian cycle in the simplex}
\label{Hopf}
\end{center}
\end{figure}

\begin{example}\label{exthetaC}
In Fig.~\ref{Borromeo} we show the link corresponding to a Hamiltonian theta-subgraph in the cube.
It is the Borromean rings. The manifold $N(P,\widetilde{\Lambda}_{\Gamma})$ has a Euclidean structure, and, as we will seen in 
Example~\ref{ex:BorAs}, the complement $S^3\setminus C_\Gamma$ has a hyperbolic structure.
This example corresponds to \cite[Example 13.1.5]{T02}, where the complement to the Borromean rings in $S^3$
is glued from the cube with $6$ segments on its boundary deleted. This cube can be divided into $8$ cubes 
in such a way that the halves of the segments form triples of edges complementary to the Hamiltonian theta-subgraph. 
These $8$ cubes are glued exactly as in $N(P,\Lambda_{\Gamma})\simeq S^3$. The author is grateful to Vladimir
Gorchakov for pointing out a connection between these examples.
\end{example}
\begin{figure}
\begin{center}
\includegraphics[width=\textwidth]{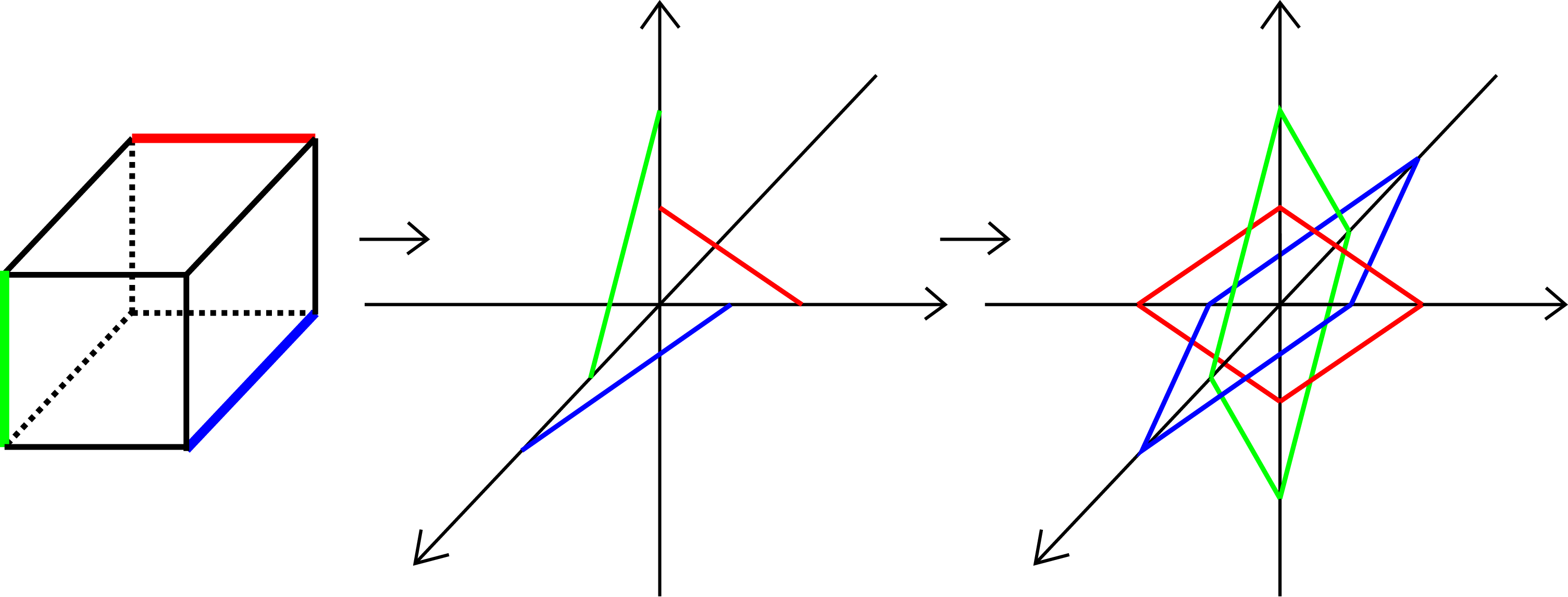}
\caption{The Borromean rings corresponding to a Hamiltonian theta-subgraph in the cube}
\label{Borromeo}
\end{center}
\end{figure}
\begin{remark}
In algebraic topology the Borromean rings are associated to the triple {\it Massey product} -- an operation producing
a new cohomology class from three classes with trivial pairwise products. In particular, for the complement of the Borromean rings 
the triple Massey product is defined and non-zero \cite{M68}. The product of the $1$-cochains 
dual to the $3$ rings via Alexander duality is zero, while the triple Massey product is non-zero. 
As we will see in Example~\ref{ex:BorAs}, toric
topology associates to the Borromean rings also a compact $12$-dimensional manifold with a nontrivial triple Massey product.
It is the moment-angle manifold of the $3$-dimensional associahedron.
\end{remark}

\subsection{Links in manifolds different from the sphere}
The link $C_\Gamma\subset S^3$ can be generalized as follows.
\begin{proposition}{\rm (see \cite[Proposition 4.15]{E26})}
For a linearly independent vector-coloring $\Lambda$ of rank $r$ of a simple $3$-polytope $P$ and an involution
$\tau\in\mathbb Z_2^r\setminus\{0\}$ the orbit space 
$N(P,\Lambda)/\langle\tau\rangle=N(P,\Lambda_\tau)$ is a closed manifold if and only if $\tau\ne \Lambda_j$
for any facet $F_j$ and $\tau\ne\Lambda_i+\Lambda_j+\Lambda_k$ for any vertex $v=F_i\cap F_j\cap F_k$ of $P$.
\end{proposition}

\begin{construction}[Links corresponding to an involution $\tau$]
For a linearly independent vector-coloring $\Lambda$ of rank $r$ of a simple $3$-polytopes $P$ and an involution
$\tau\in\mathbb Z_2^r$ if the orbit space $N(P,\Lambda)/\langle\tau\rangle$ is a closed manifold, then
the mapping $N(P,\Lambda)\to N(P,\Lambda)/\langle\tau\rangle$ is a $2$-sheeted branched covering with the branch set being 
the link in the manifold $N(P,\Lambda)/\langle\tau\rangle$ 
consisting of circles corresponding to edges $E_{i,j}=F_i\cap F_j$ of $P$ 
such that $\Lambda_i+\Lambda_j=\tau$. Denote this set of edges $M_\tau$. This is a matching, for otherwise in
some vertex $F_i\cap F_j\cap F_k$ we have $\tau=\Lambda_i+\Lambda_j=\Lambda_j+\Lambda_k$ and $\Lambda_i=\Lambda_j$.
The edge $E_{i,j}$ connecting the vertices $F_i\cap F_j\cap F_k$
and $F_i\cap F_j\cap F_l$ corresponds to $2^{r-2-c_{i,j}}$ circles, where $c_{i,j}=2$, if $\Lambda_i$, $\Lambda_j$,
$\Lambda_k$ and $\Lambda_l$ are linearly independent, and $c_{i,j}=1$, otherwise. Each circle is glued of $2^{c_{i,j}}$ copes
of $E_{i,j}$. The details see in \cite[Proposition 4.31]{E26}). We will denote this link $C_\tau$. 

If $M_\tau=\varnothing$, then $\tau$ acts freely,
the mapping $N(P,\Lambda)\to N(P,\Lambda)/\langle\tau\rangle$ is a $2$-sheeted covering, and $C_\tau=\varnothing$.
\end{construction}
\begin{remark}
If $N(P,\Lambda_\tau)$ is orientable, then its Kneser-Milnor decomposition and Thurston geometric decomposition
can be explicitly described using \cite[Lemmas 10.4 and 10.6]{E24} and \cite[Theorems 3.12 and 4.12]{E22a}. 
Similarly for $N(P,\Lambda)$.
\end{remark}

\section{Criterion when the link $C_\tau$ is hyperbolic}

\subsection{Right-angled polytopes in $\mathbb L^3$}
\begin{definition}
A {\it $k$-belt} of a $3$-polytope 
is a cyclic sequence of $k$ facets such that facets are adjacent if and only if they are successive and no
three facets have a common vertex. 

A simple $3$-polytope is {\it flag} if $P\ne\Delta^3$ and $P$ has no $3$-belts.

A simple $3$-polytope is {\it almost Pogorelov}, if it is flag any $4$-belt surrounds a quadrangle. 

A simple $3$-polytope is {\it Pogorelov}, if it is flag and has no $4$-belt.
\end{definition}
It can be shown that if an almost Pogorelov polytope $P$  has adjacent quadrangles, then $P$ is the $4$-prism (cube) or the 
$5$-prism.

It follows from results by A.V.~Pogorelov and E.M.~Andreev that a 
$3$-polytope is combinatorially equivalent to a~compact right-angled hyperbolic $3$-polytope if and only if
it is a Pogorelov polytope (see more details in \cite{E19}). 

\begin{definition}
Denote by $\mathcal{R}$ the family of combinatorial polytopes realizable in the Lobachevsky space $\mathbb L^3$ 
as polytopes of finite volume with right dihedral angles.
\end{definition}
Each finite vertex of a right-angled polytope of finite volume in $\mathbb L^3$ 
has valency $3$, and ideal vertices have valency $4$. Using Andreev's theorem it can be shown that cutting off ideal vertices 
defines a bijection between the family $\mathcal{R}$ and the family of almost Pogorelov polytopes different from the cube and the $5$-prism (see
\cite[Theorem 10.3.1]{DO01} and \cite[Theorem 6.5]{E19}). Moreover, all quadrangles of the resulting
polytope arise from ideal vertices. 
\subsection{Contraction of a matching}
In this section we will give a criterion when the plane graph obtained by contraction of edges of a matching  $M$
in the graph of a simple $3$-polytope $P$ is the graph of a polytope in $\mathcal{R}$.
%In particular, if the matching $M_\Gamma$ is complementary to $\Gamma$, where $\Gamma$ 
%is a Hamiltonian cycle, a Hamiltonian theta-subgraph, or a Hamiltonian $K_4$-subgraph on a simple $3$-polytope $P$
%satisfies this criterion, then the complement $S^3\setminus C_\Gamma$ is glued from the copies of $P$ and is hyperbolic.
\begin{definition}
For a $k$-belt $\mathcal{B}=(F_{i_1},\dots, F_{i_k})$ define the matching 
$M_\mathcal{B}=\{F_{i_1}\cap F_{i_2},\dots, F_{i_{k-1}}\cap F_{i_k}, F_{i_k}\cap F_{i_1}\}$. We will say that the belt has
$l$ edges in $M$, if $|M_\mathcal{B}\cap M|=l$.
For a matching $M$
in a simple $3$-polytope $P$ denote by $G(P)/M$ the plane graph obtained by contraction of all the edges in~$M$.
Also denote by $P_M$ the polytope obtained by cutting off all the edges of $P$ by different planes.
\end{definition}
\begin{theorem}\label{thMH}
Let $M$ be a matching in a simple $3$-polytope $P$. If $P=\Delta^3$ or $\Delta^2\times I$, then $G(P)/M$ 
is not a graph of a polytope in $\mathcal{R}$. If $P\ne \Delta^3, \Delta^2\times I$, then $G(P)/M$ is 
a graph of a polytope in $\mathcal{R}$ if and only if the following two conditions hold:
\begin{enumerate}
\item any $3$-belt of $P$ has at least two edges in $M$;
\item any $4$-belt of $P$ has at least one edge in $M$;
\end{enumerate} 
\end{theorem}
\begin{remark}
In particular, if $P$ is a Pogorelov polytope, then for any matching $M$ the graph
$G(P)/M$ is a graph of a polytope in $\mathcal{R}$ 
(this follows also from  \cite[Theorems 6.5 and 11.6]{E19}). The contraction of edges of right-angled hyperbolic polytopes producing right-angled polytopes of finite volume is discussed in \cite{BD25}. In particular, the antiprism $A_n$ is obtained by a~contraction of a perfect matching of the L\"obell polytope ($n$-barrel) $L_n$. The inverse operation corresponds to the {\it hyperbolic Dehn filling}.
\end{remark}
\begin{proof}[Proof of Theorem~\ref{thMH}]
The graph $G(P)/M$ is a graph of a polytope in $\mathcal{R}$ if and only if the polytope $P_M$
is an almost Pogorelov polytope different from the cube and the $5$-prism and all the quadrangles of 
$P_M$ arise from edges in $M$ (we will call such quadrangles {\it $M$-quadrangles}). 
That is, $P_M$ should be flag and any $4$-belt should 
surround an $M$-quadrangle. In particular, if $\mathcal{B}$ is a $3$-belt in $P$,
then it should have at least two edges in $M$. For otherwise, in $P_M$ there is either a $3$-belt, or a $4$-belt not 
corresponding to an edge in $M$. Similarly,  any  $4$-belt in $P$ should have at least one edge in $M$.

The simplex $P=\Delta^3$ has two combinatorially different 
matchings. The first matching has one edge and $P_M=\Delta^2\times I$ has a $3$-belt. The second matching has two edges, and 
$P_M=I^3$. 

The prism $P=\Delta^2\times I$ has a unique $3$-belt $\mathcal{B}$. If $G(P)/M$ is a graph of a polytope in 
$\mathcal{R}$, then by the above argument at least two
edges from $M(\mathcal{B})$ belong to $M$. If $M$ has exactly two edges, then $P_M$ is the $5$-prism. A contradiction.
Otherwise, $M$ consists of three edges, and $M(\mathcal{B})=M$. Then $P_M$ is the $6$-prism, 
which is not almost Pogorelov. A contradiction.

Thus, we have proved the theorem in one direction. Now let $P\ne \Delta^3, \Delta^2\times I$ satisfy the condition of the theorem.

Let $P_M$ have a $3$-belt $\mathcal{B}$. 
If it contains no $M$-quadrangles, then $P$ has a $3$-belt  $\mathcal{B}'$ with no edges in $M$.
A contradiction. If it contains an $M$-quadrangle, then this quadrangle corresponds to some edge 
$F_i\cap F_j$ of $P$ and the other two faces of $\mathcal{B}$ correspond to facets $F_k$ and $F_l$ different from $F_i$ and $F_j$ such that $F_k\cap F_l$ is an edge of $P$. Then $F_i\cap F_k\cap F_l$ is a vertex, for otherwise $(F_i,F_k,F_l)$ 
is a $3$-belt without edges in $M$. Then $F_i$ is
a triangle. Similarly, $F_j$ is a triangle. Then $P=\Delta^3$, which is a contradiction. Thus, $P_M$ has no $3$-belts. 
In particular, any $M$-quadrangle is surrounded by a $4$-belt.

Let $P_M$ have a $4$-belt  $\mathcal{B}$ not surrounding an $M$-quadrangle.
If it contains no $M$-quadrangles, then in $P$ it corresponds to a cyclic sequence of $4$ facets 
$\mathcal{B}'=(F_i,F_j,F_k,F_l)$ such that successive facets are adjacent and  $M(\mathcal{B}')\cap M=\varnothing$.
If $F_i\cap F_k\ne\varnothing$, then $F_i\cap F_k\in M$. We have $F_i\cap F_k\cap F_j$ is a vertex, for otherwise 
$(F_i,F_k,F_j)$ is a $3$-belt with one edge in $M$. Similarly,  $F_i\cap F_k\cap F_l$ is a vertex. Then $\mathcal{B}$ is 
a~$4$-belt corresponding to the edge $F_i\cap F_k\in M$.
A contradiction. By the same argument  $F_j\cap F_l=\varnothing$. Then $\mathcal{B}'$ is a $4$-belt without edges in $M$. A contradiction.

If $\mathcal{B}$ contains exactly one $M$-quadrangle, then this quadrangle corresponds to some edge 
$F_i\cap F_j$ of $P$ and the other three faces of $\mathcal{B}$ correspond to facets $F_k$, $F_l$, and $F_r$ of $P$
such that $F_k\cap F_l$ and $F_l\cap F_r$ are edges not in $M$.
There are two possibilities. Either $F_i\cap F_j\cap F_k$ and $F_i\cap F_j\cap F_r$ are vertices of $F_i\cap F_j$, or $F_i=F_k$ 
and $F_j=F_r$. 

In the first case consider the cyclic sequence of facets $(F_i,F_k, F_l,F_r)$.  If $F_k\cap F_r\ne \varnothing$,
then $F_i\cap F_k\cap F_r$  is a vertex, for otherwise $(F_i,F_k,F_r)$ is a $3$-belt with one edge in $M$. Then $F_i$
is a triangle. Similarly, $F_j$ is a triangle. Then $P=\Delta^3$. A contradiction. Thus, $F_k\cap F_r=\varnothing$.
If $F_i\cap F_l\ne\varnothing$, then $F_i\cap F_l\cap F_k$ is a vertex, for otherwise $(F_i,F_l,F_k)$ is a $3$-belt
with at most one edge in $M$. Similarly, $F_i\cap F_l\cap F_r$ is a vertex. Then $F_i$ is a quadrangle. 
If $F_j\cap F_l\ne\varnothing$, then similarly $F_j$ is a quadrangle, and $P=\Delta^2\times I$. A contradiction. Thus,
$F_j\cap F_l=\varnothing$ and $(F_l,F_k,F_j,F_r)$ is a $4$-belt with no edges in $M$. A contradiction. Thus, $F_i\cap F_l=\varnothing$ and $(F_i,F_k,F_l,F_r)$ is a $4$-belt with no edges in $M$. A contradiction. Thus, the first case is impossible.

In the second case $F_i\cap F_j\cap F_l=\varnothing$, since $\mathcal{B}$ is a belt. Then $(F_i,F_j,F_l)$ is a $3$-belt
with only one edge in $M$. A contradiction. Thus, the case when $M$ contains exactly one $M$-quadrangle is impossible. 

If $\mathcal{B}$ contains two $M$-quadrangles (corresponding to edges $E_1$ and $E_2$ of $P$),
then they are not adjacent and the other two facets of $\mathcal{B}$
correspond to facets $F_i$ and $F_j$ of $P$. It is not possible that $E_1,E_2\subset F_i\cap F_j$.
Thus, there are two possibilities: either one of these edges is $F_i\cap F_j$ and the other edge intersects $F_i$ and $F_j$
at vertices, or both edges intersect $F_i$ and $F_j$ at vertices.
In the first case, let the other edge be $F_k\cap F_l$. Then $F_i\cap F_j\cap F_k\ne\varnothing$, for otherwise $(F_i,F_j,F_k)$
is a $3$-belt with only one edge in $M$. Then $F_k$ is a triangle. Similarly, $F_l$ is a triangle, and $P=\Delta^3$. A contradiction.
In the second case let $E_1=F_k\cap F_l$ and $E_2=F_r\cap F_s$.  
Assume that $F_i\cap F_j\ne \varnothing$. Then $F_i\cap F_j\cap F_k\ne\varnothing$ and $F_k$ is a triangle, for otherwise 
$(F_i,F_j,F_k)$ is a $3$-belt with exactly one edge $F_i\cap F_j$ in $M$. Similarly, $F_l$ is a triangle. A contradiction. Thus, 
$F_i\cap F_j=\varnothing$. Consider the line segments $I_1$ in $F_i$ connecting the vertices $F_i\cap E_1$ and $F_i\cap E_2$. 
Similarly take $I_2$ in $F_j$. Then $(I_1,E_1,I_2,E_2)$ is a simple curve on the boundary of $P$. It divides $\partial P$ into two connected components. Let $F_k$ and $F_r$ lie in the closure of one component, and $F_l$ and $F_s$ lie in the closure
of the other.
If $F_k=F_r$, then this facet is a quadrangle, and $F_l\ne F_s$. If $F_l\cap F_s\ne\varnothing$, then $F_i\cap F_l\cap F_s$
is a vertex, for otherwise $(F_i,F_l,F_s)$ is a $3$-belt with at most one edge in $M$. Similarly, $F_j\cap F_l\cap F_s$ is a vertex.
Then $F_l$ and $F_s$ are quadrangles, and $P=\Delta^2\times I$. A contradiction. Thus, 
$F_k\ne F_r$. Similarly, $F_l\ne F_s$. Then $(F_i,F_l,F_j,F_r)$ is a $4$-belt with no edges in $M$.  
A contradiction. Thus, we have considered all possible cases and the theorem is proved.
\end{proof}
\subsection{Hyperbolic links $C_\tau$}
\begin{definition}
We call a link $C$ in a $3$-manifold {\it hyperbolic} if its complement has a complete hyperbolic structure of finite volume.
\end{definition}

\begin{theorem}\label{thHL}
Let $\Lambda$ be a linearly independent vector-coloring $\Lambda$ of rank $r$ of a simple $3$-polytope $P$ 
and $\tau\in\mathbb Z_2^r\setminus\{0\}$ be an involution such that $N(P,\Lambda)/\langle\tau\rangle$
is a closed topological manifold and $M_\tau\ne\varnothing$. Then the following conditions are equivalent.
\begin{enumerate}
\item $P$ is flag and any its $4$-belt has at least one edge in $M_\tau$.
\item  $G(P)/M_\tau$ is a graph of a polytope in $\mathcal{R}$.
\item The link $C_\tau$ is hyperbolic.
\end{enumerate}
\end{theorem}
\begin{proof}
The implication (1)$\Rightarrow$ (2) follows from Theorem~\ref{thMH}.

\begin{lemma}\label{Lcomp}
If $G(P)/M_\tau$ is a graph of a polytope $Q\in \mathcal{R}$, then the complement 
$N(P,\Lambda_\tau)\setminus C_\tau$ is homeomorphic to the hyperbolic manifold 
$\mathbb L^3/{\rm Ker}\,\varphi_{\Lambda_\tau}$
defined for the induced vector-coloring $\Lambda_\tau$ of $Q$ by Construction~\ref{VMC}.
\end{lemma}
\begin{proof}
The proof follows from Lemma~\ref{Rcusp} and the fact that the 
manifold $N(\widehat Q,\widehat{\Lambda_\tau})$ is homeomorphic to the manifold 
$N(P,\Lambda_\tau)$ without tubular neighbourhoods of circles corresponding to the edges in $M_\tau$. 
\end{proof}
Thus, (2)$\Rightarrow$ (3). 

If (2) holds, then by Theorem~\ref{thMH} we have $P\ne\Delta^3$ and any $4$-belt of $P$
has at least one edge in $M_\tau$. Also $P$ has no $3$-belts. For otherwise, if a $3$-belt $(F_i,F_j,F_k)$ has two edges in 
$M_\tau$, say $F_i\cap F_j$ and $F_j\cap F_k$,
then $\tau=\Lambda_i+\Lambda_j=\Lambda_j+\Lambda_k$ and $\Lambda_i=\Lambda_k$, which is a contradiction.
Thus, $P$ is flag. Hence, (2)$\Rightarrow$ (1).

Now assume that (3) holds, but (1) does not hold. If $P=\Delta^3$, then either the 
vectors $\Lambda_1$, $\Lambda_2$, $\Lambda_3$, $\Lambda_4$ are linearly independent, or $\Lambda_1$, 
$\Lambda_2$, $\Lambda_3$ are linearly independent, and $\Lambda_4=\Lambda_1+\Lambda_2+\Lambda_3$. In both cases 
without loss of generality we may assume that $\tau=\Lambda_1+\Lambda_2$. In both cases $N(P,\Lambda_\tau)\simeq S^3$.
In the first case $N(P,\Lambda)\simeq S^3$, and $C_\tau$ is a trivial circle.
We have $N(P,\Lambda_\tau)\setminus C_\tau\simeq D^2\times S^1$ and 
this manifold has no complete hyperbolic structure of finite volume (since the boundary torus is not incompressible, see 
\cite[Proposition D.3.18(1)]{BP92}).
In the second case $N(P,\Lambda)\simeq \mathbb RP^3$,  $C_\tau$ is the Hopf 
link consisting of two trivial circles linked together exactly once. 
We have $N(P,\Lambda_\tau)\setminus C_\tau\simeq T^2\times I$ and this manifold
is known to have no complete hyperbolic structure of finite volume (see \cite[p. 359]{T82}). Thus, $P\ne\Delta^3$.

Denote $M=M_\tau$ for short. Consider the polytope $P_M$ obtained from $P$ by cutting off all the edges of $M$. 
The space $N(P,\Lambda_\tau)\setminus C_\tau$ is homeomorphic to the interior of $N(P_M,\Lambda_M)$, 
where $\Lambda_M$ is a vector coloring of rank $(r-1)$ (as $\Lambda_\tau$)
such that $\Lambda_M(F_i)=\Lambda_\tau(F_i)$ for "old" facets $F_i$, and $\Lambda_M(G_i)=0$ for  quadrangles $G_i$ 
corresponding to edges in $M$. Denote 
the set of these quadrangles $S$.
Each quadrangle in $M$ corresponds to a family of tori in the real moment-angle manifold
$\mathbb R\mathcal{Z}_{P_M}$. If we delete from $\mathbb R\mathcal{Z}_{P_M}$ all the tori corresponding to these 
quadrangles, we obtain a disjoint union of spaces homeomorphic to the interior of the manifold $N(P_M,O_M)$, 
defined by the vector-coloring $O_M$ of rank $m$ (the number of facets in $P$) that sends ``old'' facets $F_i$
of $P_M$ to the basis vectors $e_i\in\mathbb Z_2^m$ and quadrangles from $S$ to $0$. 
Moreover, $N(P_M,O_M)=\mathbb R\mathcal{Z}_{P_M}/H$, 
where the subgroup $H$ is generated by vectors corresponding to the quadrangles in  $S$.
Then $N(P_M,O_M)\to N(P_M,\Lambda_M)$ is a finite-sheeted covering. Indeed, $N(P_M,\Lambda_M)$
is an orbit space of the action of the finite subgroup $H$, which is the kernel of
the mapping $\psi\colon\mathbb Z_2^m\to \mathbb Z_2^{r-1}$ defined by $\Lambda_\tau$, on $N(P_M,O_M)$.
This action is free, since the stabilizer of a point $[x,a]\in N(P_M,O_M)$ is $\langle O_M(F_i)\colon x\in F_i\rangle$,
and  ${\rm Ker}\,\psi\cap \langle O_M(F_i)\colon x\in F_i\rangle=\{0\}$.  
Thus, if ${\rm int}\,N(P_M,\Lambda_M)$ is hyperbolic, then ${\rm int}\,N(P_M,O_M)$ is an orientable hyperbolic manifold.

If $P$ has a $3$-belt $\mathcal{B}=(F_i,F_j,F_k)$, 
then by the above argument either this belt has no edges in $M$, or it has one edge in $M$, say $F_i\cap F_j$. Consider a triangle $T$
with vertices in the midpoints of the edges $F_i\cap F_j$, $F_j\cap F_k$, and $F_k\cap F_i$, and straight 
edges inside the facets $F_i$, $F_j$ and $F_k$. 

In the first case the preimage of $T$ in $N(P_M,O_M)$
consists of $2^{m-3}\geqslant 4$ spheres with trivial tubular neighbourhoods, where each sphere is glued of $8$ copies of $T$ (see 
more details in \cite[Proposition 3.6]{E22a}). 
Cutting along all these spheres divides $N(P_M,O_M)$ into several connected 
components. Each component is the copy of a connected manifold $N(P_M',O_M')$ or $N(P_M'',O_M'')$, where $P_M$ is 
divided into polytopes $P_M'$ and $P_M''$ if we cut along the triangle
$T$ (see \cite[Corollary 2.27]{E22a}). $O_M'$ and $O_M''$ are induced vector-colorings with $O_M'(T)=O_M''(T)=0$.
But the deletion of each single sphere leaves the manifold $N(P_M,O_M)$ connected, since each polytope  $P_M'$ and $P_M''$
has at least four facets left from $P$, and hence each manifold $N(P_M',O_M')$ and $N(P_M'',O_M'')$ has at least two 
boundary spheres corresponding to $T$. In particular, any of the spheres does not bound a $3$-disk in $N(P_M,O_M)$, 
and is essential.
Then $N(P_M,O_M)$ is not hyperbolic (see \cite[Theorem 2.9]{B02}). A contradiction.

In the second case the preimage of $T$ (with a vertex cut) in $N(P_M,O_M)$
consists of $2^{m-3}\geqslant 4$ surfaces. Each surface is glued of $8$ copies of $T$ (with a vertex cut), 
and is homeomorphic to a sphere with two holes corresponding to the cut vertex. The holes correspond to
circles on two different boundary tori of $N(P_M,O_M)$, and each circle is a meridian or a parallel of the torus.
Then in $\pi_1(N(P_M,O_M))$ these two circles represent two nontrivial conjugate elements from $\pi_1(T_1)$ and $\pi_1(T_2)$ 
(remind that boundary components are $\pi_1$-injective, see 
\cite[Proposition D.3.18(1)]{BP92}). 
On the other hand, it is known that for different boundary tori
$T_1$ and $T_2$ of a complete hyperbolic $3$-manifold of finite volume and any $g\in \pi_1(N(P_M,O_M))$ we have
$g^{-1}\pi_1(T_1)g\cap \pi_1(T_2)=\{1\}$ (see \cite[Lemma 3.4(2)]{F11}. A contradiction. Thus, $P$ is flag.

Since (1) does not hold, $P$ has a  least one $4$-belt $\mathcal{B}=(F_i,F_j,F_k, F_l)$ 
with no edges in $M$.  Consider a ``quadrangle'' $T$, which is the cone
with apex at the barycenter of $P$ over the closed curve consisting of straight edges 
in facets $F_i$, $F_j$, $F_k$, and  $F_l$ connecting the midpoints of the edges $F_i\cap F_j$, $F_j\cap F_k$, $F_k\cap F_l$, 
and $F_l\cap F_i$. 
The preimage of $T$ in $N(P_M,O_M)$ consists of $2^{m-4}\geqslant 2$ incompressible tori  
(see \cite[Proposition 4.23]{E22a}) disjoint from the boundary 
such that the deletion of all the tori divides $N(P_M,O_M)$ into several connected components. 
Each component is the copy of a connected manifold
$N(P_M',O_M')$ or $N(P_M'',O_M'')$, where $P_M$ is divided into polytopes $P_M'$ and $P_M''$ if we cut along the quadrangle 
$T$ (see \cite[Corollary 2.27]{E22a}). $O_M'$ and $O_M''$ are induced vector-colorings with $O_M'(T)=O_M''(T)=0$.
But the deletion of each single torus leaves the manifold $N(P_M,O_M)$ connected, since each polytope  $P_M'$ and $P_M''$
has at least five facets left from $P$,  and hence each manifold $N(P_M',O_M')$ and $N(P_M'',O_M'')$ has at least two 
boundary tori corresponding to $T$. But in a complete hyperbolic $3$-manifold of finite volume
the deletion of each incompressible tori disjoint from the boundary divides the manifold into two connected components, 
where one of this components is $T^2\times I$ with the second boundary torus being the boundary torus (cusp) 
of the manifold (see \cite[before Corollary 1.8]{H3MT}, \cite[Theorem 2.3 and before]{T82}, \cite[Section 1.6]{AFW15}). 
A contradiction. 

Thus, (3) $\Rightarrow$ (1), and the theorem is proved.
\end{proof}

\section{Parametrization of the set of hyperbolic links $C_\Gamma$}
Theorem~\ref{thHL} implies the following result.
\begin{theorem}\label{thGHL}
For a link $C_\Gamma$ corresponding to a Hamiltonian cycle, theta- or $K_4$-subgraph $\Gamma$
in a simple $3$-polytope $P$ the following conditions are equivalent.
\begin{enumerate}
\item $P$ is flag and for any $4$-belt of $P$ at least two successive facets 
belong to the same connected component of $\partial P\setminus\Gamma$.
\item  $G(P)/M_\Gamma$ is a graph of a right-angled hyperbolic $3$-polytope of finite volume.
\item The link $C_\Gamma$ is hyperbolic.
\end{enumerate}
\end{theorem}
\begin{corollary}
For any Hamiltonian subgraph $\Gamma$ in a Pogorelov polytope $P$ the link $C_\Gamma$ is hyperbolic.
\end{corollary}
\begin{definition}
If $G(P)/M_\Gamma$ is a graph of a $3$-polytope, we will denote this polytope $Q_\Gamma$
\end{definition}
For a hyperbolic link $C_\Gamma$ its complement it glued of $2^k$ copies of the right-angled polytope $Q_\Gamma$,
where $k=2$ for the Hamiltonian cycle, $k=3$ for the theta-subgraph, and $k=4$ for the $K_4$-subgraph.
\begin{example}\label{ex:BorAs}
The theta-subgraph in the cube $I^3$
 from Example~\ref{exthetaC} satisfies conditions of Theorem~\ref{thGHL}. The polytope $Q_\Gamma$ is 
a right-angled $3$-gonal bipyramid with $2$ proper and $3$ ideal vertices.  
The complement of the Borromean rings $C_{\Gamma}$
is glued of $8$ copies of $Q_\Gamma$. 
 As it was mentioned above this representation of the complement $S^3\setminus C_{\Gamma}$ 
can be extracted from \cite[Example 13.1.5]{T02}. Also this representation is equivalent 
to one from \cite[Examples 7 and 9]{MR22} and \cite[Section 1.1.2]{M24}. Recently in \cite{VE25} it was proved that $Q_\Gamma$
has the smallest volume among all hyperbolic right-angled $3$-polytopes of finite volume.
The almost Pogorelov polytope associated to the $3$-gonal bipyramid is the $3$-dimensional
associahedron (Stasheff polytope) $As^3$. As is it mentioned in \cite[Remark after Example 4.9.4]{BP15} the 
$12$-dimensional moment-angle manifold $\mathcal{Z}_{As^3}$ has a nontrivial triple Massey product 
(it follows also from \cite[Theorem 6.1.1]{DS07}). In \cite{L19}, nontrivial triple Massey products were constructed in cohomology
of moment-angle manifolds of more general graph-associahedra. 
\end{example}

Now let us describe the family of hyperbolic links $C_\Gamma$ in terms of
right-angled hyperbolic $3$-polytopes of finite volume $Q_\Gamma$.  On the polytope $Q_\Gamma$ the subgraph $\Gamma$ induces the following structure.
\begin{definition}
A cycle in a graph $G$ is called {\it Eulerian} if it passes each edge of $G$ once (it may pass one vertex many times).

An {\it Eulerian theta-subgraph} in a graph $G$ consists of three paths connecting two different vertices.  Each edge
of $G$ belongs to exactly one path and is traversed exactly once.

An {\it Eulerian $K_4$-subgraph} in a graph $G$ consists of $4$ different vertices and $6$ paths connecting these vertices. 
Each pair of vertices is connected by a path. Each edge of $G$ belongs to exactly one path and is traversed exactly once.

For short we will call by an {\it  Eulerian subgraph} an Eulerian cycle, an  Eulerian theta-subgraph, or an
Eulerian $K_4$-subgraph and denote it $\gamma$.

An Eulerian subgraph in a plane graph is {\it nonselfcrossing} if each path traverses any its interior vertex 
(that is, different from its ends) by edges successive in the cyclic order around this vertex (it may visit a vertex more than once).
\end{definition}
The following result is straightforward from the definition.
\begin{proposition}\label{PEg}
Let $\Gamma$ is a Hamiltonian subgraph in a simple $3$-polytope $P$. Then it induces a nonselfcrossing
Eulerian subgraph $\gamma_\Gamma$ in $G(P)/M_{\Gamma}$.
\end{proposition}
\begin{remark}
The transition from a Hamiltonian cycle $\Gamma$ in a cubic graph $G$ to the nonselfcrossing Eulerian cycle
in the graph $G/M_{\Gamma}$ was used in \cite{BP17} to give a new characterisation of
cubic Hamiltonian graphs having a~perfect matching. 
\end{remark}

\begin{definition}
Let us call a connected plane graph $G$ {\it admissible}, if it has $0$, $2$ or $4$ vertices of valency $3$ and all 
the other vertices of valency $4$.
\end{definition}
The vertices of an Eulerian subgraph $\gamma$ in an admissible plane graph $G$ are exactly $3$-valent vertices of $G$.
The plane graph $G(P)/M_{\Gamma}$ corresponding to a Hamiltonian subgraph in a simple $3$-polytope $P$
is admissible. 
\begin{question}
To describe the class of admissible plane graphs arising as $G(P)/M_{\Gamma}$ 
for Hamiltonian subgraphs in simple $3$-polytopes.
\end{question}
We give a partial answer to this question for Hamiltonian cycles in Corollary \ref{corch4}.  

The graph $G(P)/M_{\Gamma}$ may not be a graph of a $3$-polytope. For example, if $\Gamma$ is a Hamiltonian cycle and 
$P$ contains a triangle, then one of the faces of $G(P)/M_{\Gamma}$ is a bigon. For a combinatorially unique Hamiltonian cycle
$\Gamma$ in  the cube $I^3$ two facets of $G(P)/M_{\Gamma}$ are bigons. 

\begin{question}
To characterise Hamiltonian subgraphs $ \Gamma$ in simple $3$-polytopes such that 
$G(P)/M_{\Gamma}$ is a graph of a $3$-polytope. 
\end{question}

\begin{construction}\label{ADC}
Any nonselfcrossing Eulerian subgraph $\gamma$ in an admissible plane graph $G$ 
produces a plane cubic graph $G_\gamma$ with a Hamiltonian cycle $\Gamma_\gamma$ such that 
$G=G_\gamma/M_{\Gamma_\gamma}$. For this substitute each $4$-valent vertex of $G$ by 
two vertices connected by an edge in such a way that each pair of successive edges of $\gamma$ at this vertex 
is incident to the same vertex of the new edge. If $G_\gamma$ is a graph of simple $3$-polytope, we will denote this polytope $P_\gamma$.
\end{construction}
\begin{remark}
For a nonselfcrossing Eulerian cycle in a $4$-valent plane graph this 
construction is a particular case of a construction of the graph $A(E)$ used in the proof of \cite[Theorem 14]{K68b}. 
\end{remark}

\begin{proposition}\label{PGg}
Let $\gamma$ be a nonselfcrossing Eulerian subgraph $\gamma$ in an admissible plane graph $G$,
where $G=G(Q)$ for a $3$-polytope $Q$. Then $G_\gamma$ is a graph of a simple $3$-polytope.
\end{proposition}
\begin{proof}
Each face of $G$ is bounded by a simple edge cycle and if two such boundary cycles intersect, then
at a vertex or an edge. Then in $G_\gamma$ any face is also bounded by a simple edge-cycle. If two such cycles
intersect, then by a finite set of disjoint edges, since $G_\gamma$ is cubic.
After shrinking these edges correspond to disjoint vertices and edges. Therefore, there is only one edge. Thus, by
the Steinitz theorem $G_\gamma$ is a graph of a simple $3$-polytope.
\end{proof}
\begin{definition}
For short we will denote by $C_\gamma$ 
the link $C_{\Gamma_\gamma}$ corresponding to a nonselfcrossing Eulerian subgraph $\gamma$
in a simple $3$-polytope $Q$.
\end{definition}

\begin{theorem}\label{thHLP}
Hyperbolic links $C_\Gamma$ bijectively correspond to nonselfcrossing 
Eulerian subgraphs in right-angled hyperbolic $3$-polytopes
of finite volume with $0$, $2$, or $4$ finite vertices and all the other vertices lying at infinity.
\end{theorem}
\begin{proof}
This follows directly from Theorem~\ref{thGHL} and Propositions~\ref{PEg} and~\ref{PGg}.
\end{proof}
\section{Links $C_\Gamma$ corresponding to Hamiltonian cycles}
\subsection{Existence of nonselfcrossing Eulerian cycles}
In the paper \cite{K68b} on Eulerian cycles in $4$-valent graphs a nonselfcrossing Eulerian cycle is called a {\it $\sigma$-line}.
In the paper \cite{FSW92} devoted to transformations of Eulerian trails and in the paper \cite{BFFS18} devoted to Barnette's conjecture that {\it every $3$-connected cubic planar bipartite graph is Hamiltonian},
a nonselfcrossing Eulerian cycle is called an {\it $A$-trail} ($A$ means {\it admissible}). 
As it was mentioned by D.V.~Talalaev nonselfcrossing Eulerian cycles arise in the theory of electrical 
networks \cite{BKT26}.
The following result is proved in \cite[Theorem 1]{K68a}, \cite[Theorem 10]{K68b}, and \cite[Theorem 1]{B83}. 
\begin{proposition}
Any plane $4$-valent graph has a nonselfcrossing Eulerian cycle.
\end{proposition}
The idea of the proof is very simple: first take any Eulerian cycle, and then modify it at each vertex of transversal selfintersection.
\begin{question}
Does every right-angled hyperbolic $3$-polytope
of finite volume with $2$ (or $4$) finite vertices and all the other vertices lying at infinity have 
a nonselfcrossing Eulerian theta-subgraph (or $K_4$-subgraph)?
\end{question}
\subsection{Medial graphs and ideal right-angled polytopes} 
Let $G$  be a plane graph. 
Its {\it medial graph}\index{medial graph} is a new plane graph $M(G)$ with vertices bijectively corresponding to edges of $G$. 
Its edges arise when we walk around the boundary cycle of each face. Each vertex of this cycle corresponds
to an edge of $M(G)$ connecting the vertices corresponding to successive edges of the cycle. A medial graph is a plane $4$-valent
graph. It is known that any such a graph $\widetilde{G}$ is a medial graph of some plane graph $G$. 
(First, one can prove that the faces of $\widetilde{G}$ have a checkerboard coloring: a coloring in black and white colors
such that faces that have a common edge have different colors. Then vertices of $G$ correspond to black (or white) faces
and  each vertex of $\widetilde{G}$ corresponds to an edge connecting the vertices of $G$ corresponding to faces incident to this vertex.)

It is known that a graph $\widetilde{G}$ is a graph $G(\widetilde{P})$ of an ideal hyperbolic right-angled $3$-polytope  
$\widetilde{P}$ if and only if 
it is a medial graph of some polytope $P$ (not necessarily simple). Moreover, $P$ 
is defined uniquely up to passing to the dual polytope $P^*$ (see more details in \cite[Section 9]{E19}). Given a $4$-valent 
plane graph $\widetilde{G}$ one can determine whether it is a graph of an ideal right-angled $3$-polytope as follows. 
First build a graph $G$ such that $\widetilde{G}=M(G)$ as described above. Then check that $G$ is  
simple, any its face is bounded by a simple edge-cycle, and if two such boundary cycles  intersect, then their intersection is a 
vertex or and edge (see the Steinitz theorem~\ref{StT}).
\begin{remark}
According to \cite[Theorem 14]{K68b}, \cite{L78}, \cite[Corollary 3.5]{B87}, 
nonselfcrossing Eulerian cycles in the medial graph $M(G)$ are in bijection with spanning
trees $T$ of $G$. Namely, we add an edge $E$ of $G$ to $T$ if and only if in the corresponding vertex 
of $M(G)$ the Eulerian cycle both times traverses the pairs of edges corresponding to a face of $G$ containing $E$.
\end{remark}
\begin{remark}
It is easy to see that nonselfcrossing Eulerian cycles in a $4$-valent plane graph $G$ correspond to Hamiltonian cycles in $M(G)$.
\end{remark}

\subsection{Complement to the hyperbolic link $C_\Gamma$ corresponding to a Hamiltonian cycle}
\begin{construction}[A manifold from a checkerboard coloring]
Each ideal right-angled $3$-polytope $P$ admits a checkerboard coloring: its faces can be colored in 
black and white colors in such a way that adjacent faces have different colors (if $G(P)=M(G(Q))$, then 
black faces of $P$ correspond to vertices of $Q$ and white facets of $P$ correspond to facets of $Q$). 
Assign to white color the vector $e_1\in\mathbb Z_2^2$
and to black color $e_2\in \mathbb Z_2^2$. Then we obtain the mapping 
$\Lambda_P\colon \{F_1,\dots, F_m\}\to\mathbb Z_2^2$, $F_i\to \Lambda_i$,  from the set of facets of $P$ to $\mathbb Z_2^2$, 
and the Vesnin-Mednykh Construction \ref{VMC} gives the 
complete hyperbolic manifold $N(P)$ of finite volume glued of $4$ copies of $P$:
$$
N(P)=P\times \mathbb Z_2^2/\sim, (p,a)\sim (q,b) \text{ if and only if } p=q\text{ and }a-b\in\langle \Lambda_i\colon p\in F_i\rangle. 
$$  
In this formula the ideal vertices are not assumed to belong to $P$. 
\end{construction} It was proved in \cite{E22b} that the family of manifolds 
$\{M(P)\}$, where $M(P)$ is the double of the manifold obtained from $N(P)$ by adding the boundary torus at each cusp, 
is cohomologically rigid over $\mathbb Z_2$, that is two manifolds from this family are homeomorphic if and only if their 
cohomology rings over $\mathbb Z_2$ are isomorphic as graded rings.

\begin{proposition}
Let $\gamma$ be a nonselfcrossing Eulerian cycle in an ideal right-angled $3$-polytope $Q$.
Then $S^3\setminus C_\Gamma$ is homeomorphic to the manifold $N(Q)$.
\end{proposition}
\begin{proof}
The proof follows from Lemma~\ref{Lcomp}.
\end{proof}
\begin{remark}
In \cite{CKP22}, a hyperbolic link $L$ is called {\it right-angled}, if $S^3\setminus L$ with
the complete hyperbolic structure admits a decomposition into ideal hyperbolic right-angled
polytopes. By construction the hyperbolic link $C_\Gamma$ corresponding to a Hamiltonian cycle is right-angled. 
\end{remark}

\begin{example}\label{ex:oct}
The octahedron is a unique right-angled polytope with the smallest number of vertices (equal to $6$). Up to combinatorial 
symmetries it has exactly two nonselfcrossing Eulerian cycles (see the proof in Fig.~\ref{A3cycles}) shown in Fig.~\ref{Octfig}. 
We also present the corresponding simple polytopes
and hyperbolic links. These links are not isotopic, but their complements are homeomorphic. %The ways to represent
%links follows the method from \cite{T25}. 
\begin{figure}[h]
\begin{center}
\includegraphics[width=\textwidth]{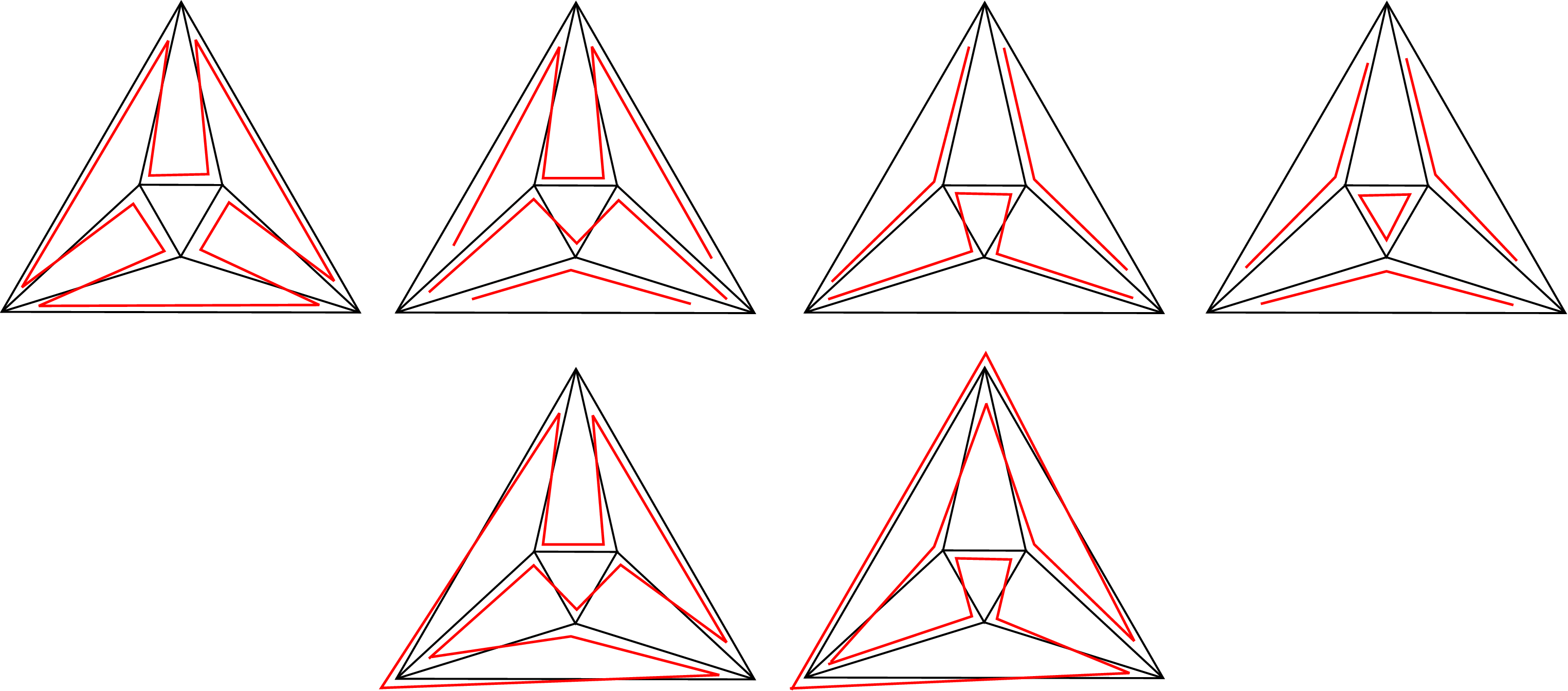}
\end{center}
\caption{Enumeration of nonselfcrossing Eulerian cycles in the octahedron}\label{A3cycles}
\end{figure}

\begin{figure}[h]
\begin{center}
\includegraphics[width=\textwidth]{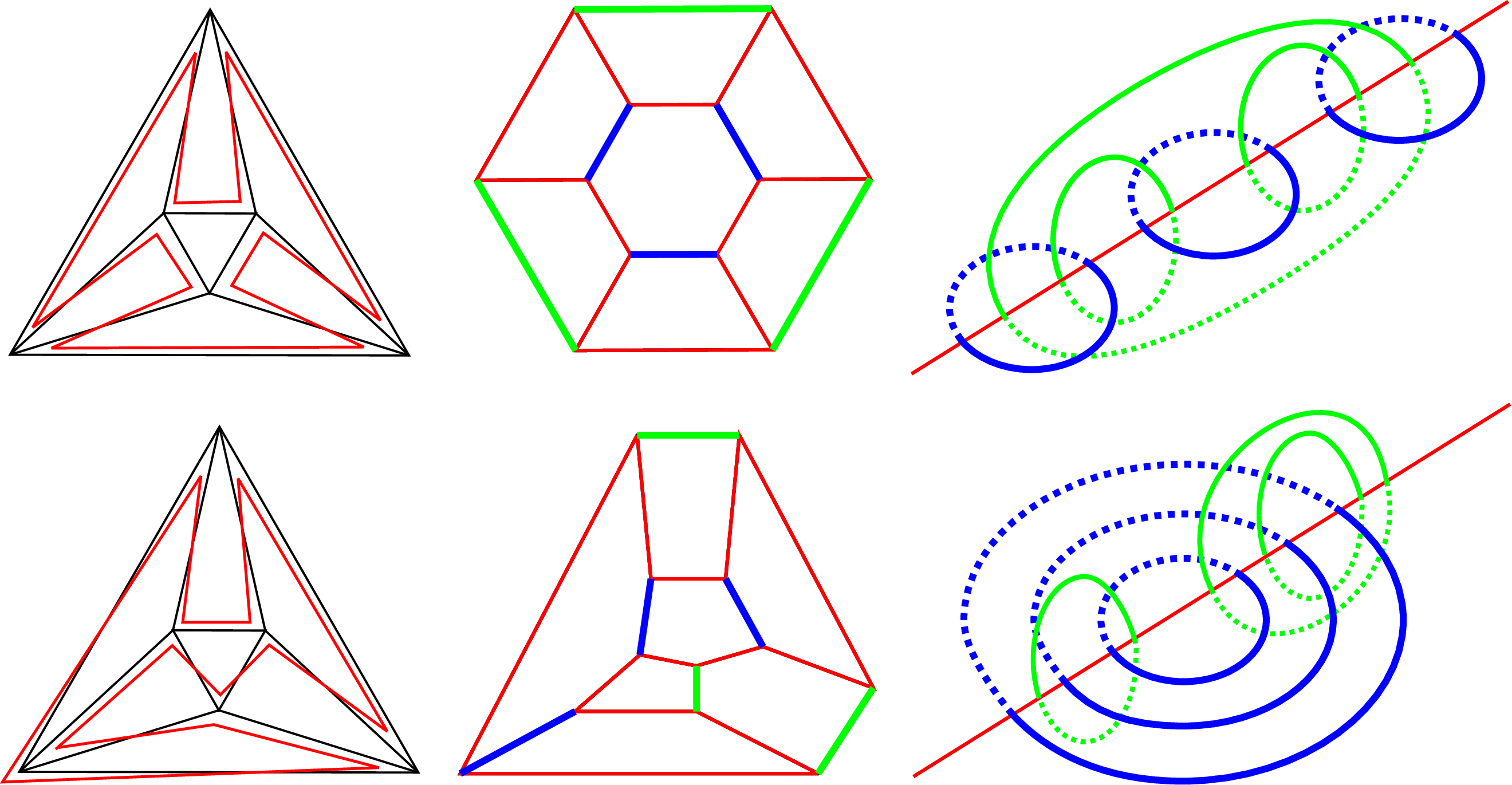}
\end{center}
\caption{Hyperbolic links corresponding to nonselfcrossing Eulerian cycles in the octahedron}\label{Octfig}
\end{figure}
\end{example}

\begin{example}\label{ex:ak}
Example~\ref{ex:oct} 
can be generalized as follows. It is known that any {\it antiprism} $A(k)$ ($A(3)$ is the octahedron) is an
ideal  right-angled $3$-polytope (see \cite{V17}). In Fig~\ref{k-anti}, we show two different nonselfcrossing Eulerian cycles on
this polytope. The hyperbolic structure on the complement to the link corresponding to the left cycle is exactly the structure defined
by W.P.~Thurston in the complement to the $(2k)$-link chain \cite[Example 6.8.7]{T02}. 
The case of $A(3)$ is also mentioned in \cite[Section 5.1]{V17}. 

\begin{figure}[h]
\begin{center}
\includegraphics[width=0.5\textwidth]{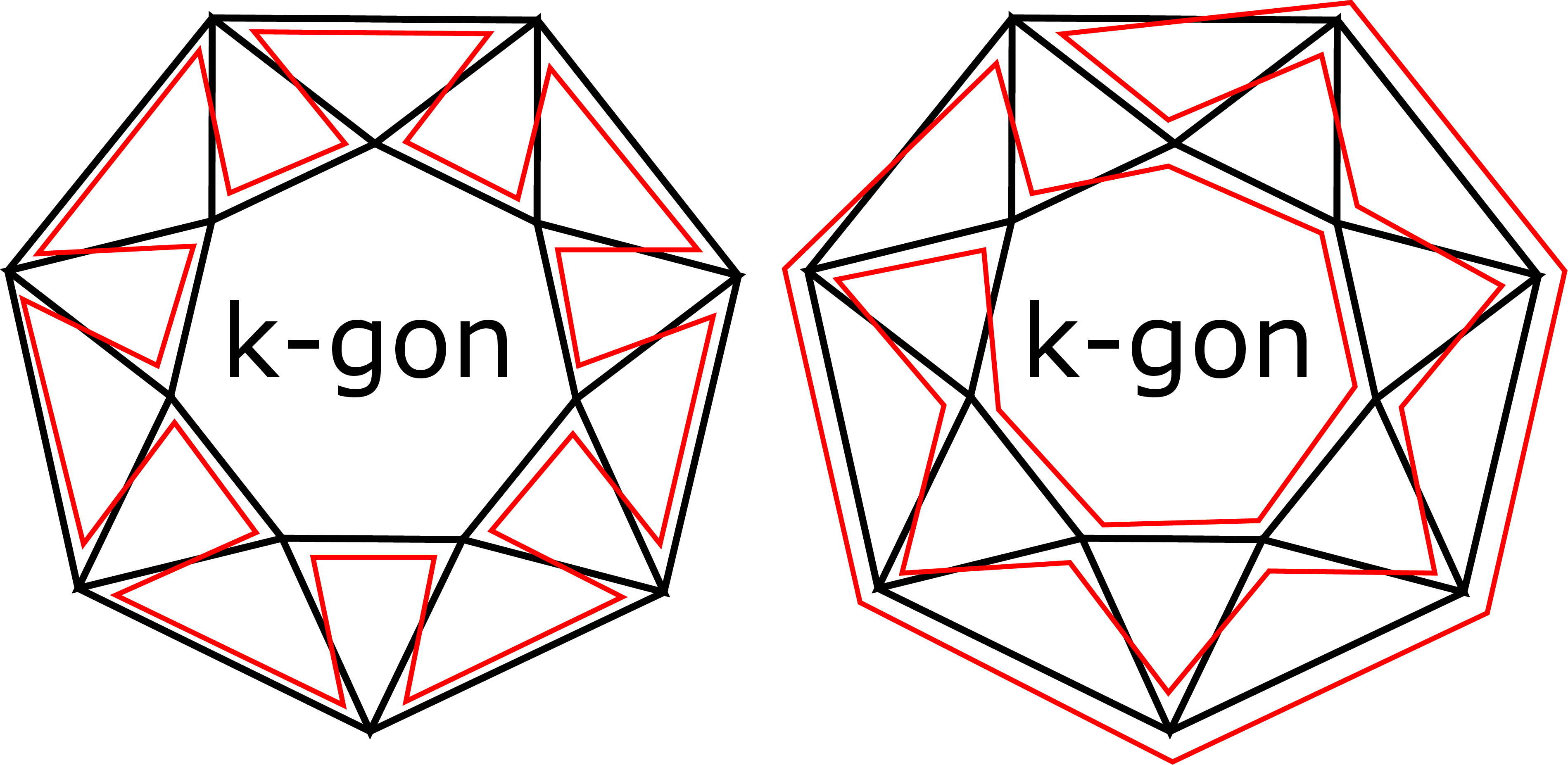}
\end{center}
\caption{Nonselfcrossing Eulerian cycles in the antiprism}\label{k-anti}
\end{figure}
\end{example}

\begin{example}\label{A4ex}
It can be shown (see Fig.~\ref{A4cycles}) that  up to combinatorial symmetries $A(4)$ has exactly $7$
nonselfcrossing Eulerian cycles  shown in Fig.~\ref{A4polytopes}.
\begin{figure}[h]
\begin{center}
\includegraphics[width=\textwidth]{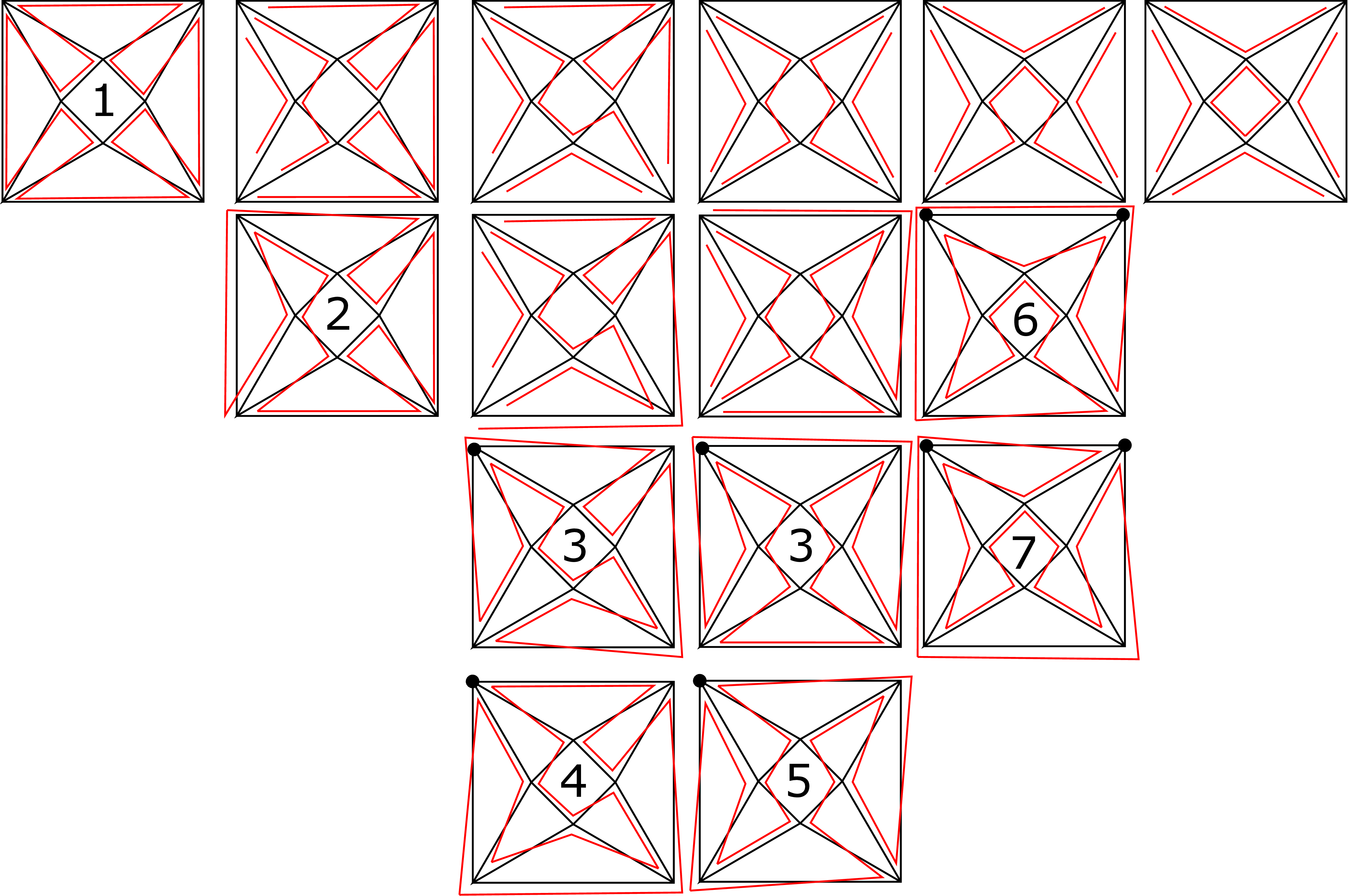}
\end{center}
\caption{Enumeration of nonselfcrossing Eulerian cycles in $A(4)$}\label{A4cycles}
\end{figure}

\begin{figure}[h]
\begin{center}
\includegraphics[width=\textwidth]{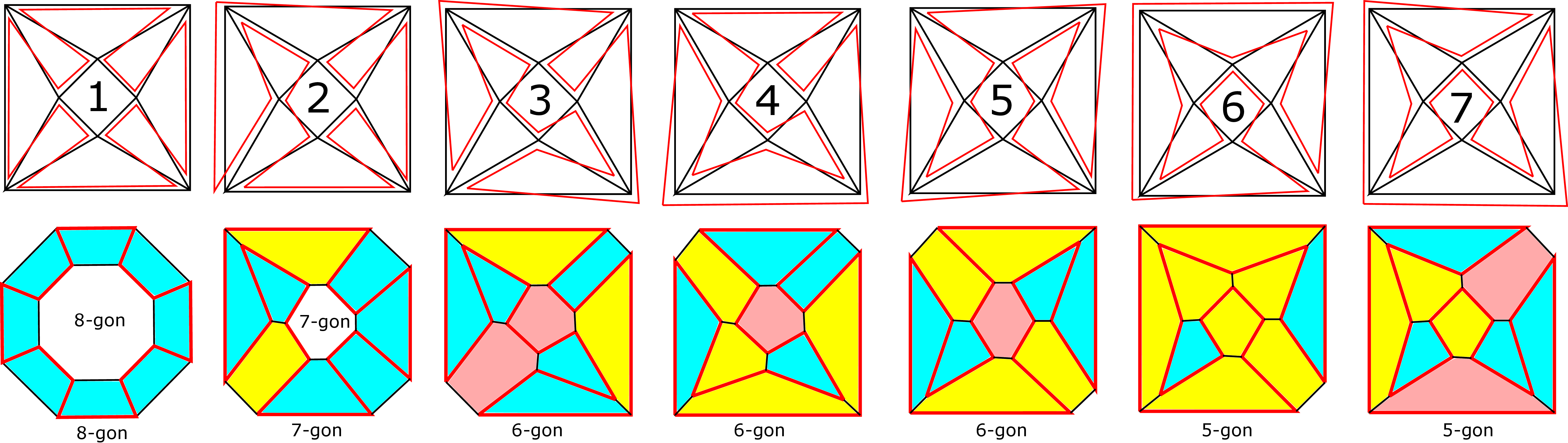}
\end{center}
\caption{Nonselfcrossing Eulerian cycles on $A(4)$ and the corresponding polytopes}\label{A4polytopes}
\end{figure}
\end{example}

\subsection{Edge-twists and construction of nonselfcrossing Eulerian cycles}

In \cite[Theorem 2.14]{V17}, on the base of results from  \cite{BGGMTW05} the following construction of all the ideal right-angled polytopes was described.

\begin{definition}
An operation of an {\it edge-twist}\index{edge-twist} is shown in Fig. \ref{e-skr}. Two edges on the left belong to one facet of an
ideal right-angled polytope and connect $4$ distinct vertices. The result is again an ideal right-angled polytope. 
Let us call an edge-twist {\it restricted} if both edges are adjacent to the same edge, that is the $4$ vertices follow each other during the round walk along the boundary of a facet.
\end{definition}
\begin{figure}
\begin{center}
\includegraphics[width=.5\textwidth]{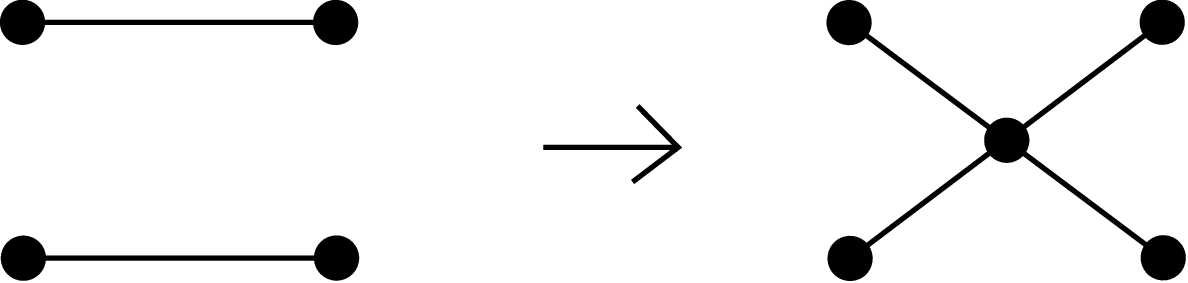}
\caption{An operation of an edge-twist}
\label{e-skr}
\end{center}
\end{figure}

\begin{theorem}[\cite{V17}]\label{V-th}
Any ideal right-angled $3$-polytope can be obtained by operations of an edge-twist from some 
$k$-antiprism $A(k)$, $k\geqslant 3$.
\end{theorem}
\begin{remark}
Operations of an edge-twist are not applicable to the octahedron, hence all the other polytopes are obtained from $k$-antiprisms, $k\geqslant 4$. 
\end{remark}

In \cite[Theorem 9.13]{E19}, this result was improved. 

\begin{theorem}[\cite{E19}]\label{I-th}
A $3$-polytope is an ideal right-angled $3$-polytope if and only if either it is a $k$-antiprism $A(k)$, $k\geqslant 3$, or it can be 
obtained from the $4$-antiprism by operations of a restricted edge-twist.
\end{theorem}

The following result is straightforward from the definitions.
\begin{proposition}
Any edge-twist transforms a nonselfcrossing Eulerian cycle to a nonselfcrossing Eulerian cycle in the new polytope. 
\end{proposition}
\begin{corollary}
The $3$-antiprism $A(3)$ (octahedron)  has exactly $2$ combinatorially different nonselfcrossing Eulerian cycles, 
the $4$-antiprism $A(4)$  has exactly $7$ combinatorially different nonselfcrossing Eulerian cycles, 
and they correspond to $7$ nonselfcrossing Eulerian cycles (perhaps some of them are combinatorially equivalent) 
in any polytope different from antiprisms,  and any antiprism $A(k)$ has at least $2$ combinatorially different cycles.
\end{corollary}
\begin{proof}
This follows from Examples~\ref{ex:oct}, \ref{ex:ak}, and \ref{A4ex}.
\end{proof}
%\begin{remark}
%The existence of a nonselfcrossing Eulerian cycle in any plane $4$-valent graph follows from 
%\cite[Theorem 1]{K68a} and \cite[Theorem 10]{K68b}.
%As was mentioned by A.A.~Gaifullin is also can be proved as follows.
%Since each vertex of $P$ has even valency, it has a Eulerian cycle. Then we can deform this
%cycle at bad vertices. If it has a transversal self-crossing, then one of the two ways to change this crossing leaves the cycle 
%connected.
%\end{remark}

\begin{question}
To enumerate all combinatorially different nonselfcrossing Eulerian cycles in any ideal right-angled $3$-polytope. To find estimates for their number. Similarly for nonselfcrossing Eulerian theta-subgraphs and $K_4$-subgraphs in hyperbolic right-angled $3$-polytopes of finite volume with $2$ and $4$ finite vertices.
\end{question}

\subsection{Transformations of nonselfcrossing Eulerian cycles}
\begin{definition}
We will call two edges $E_1$ and $E_2$ of a simple $3$-polytope $P$ not lying in a Hamiltonian cycle $\Gamma$ {\it conjugated}, 
if each edge intersects both components of the complement in $\Gamma$ to the vertices of the other edge 
(in other words, if $\Gamma\cup E_1\cup E_2$ is homeomorphic to the full graph $K_4$ on four vertices). We call two vertices
of an ideal right-angled $3$-polytope $Q$ {\it conjugated along the nonselfcrossing Eulerian cycle $\gamma$}, if 
the corresponding edges of $P_\gamma$ are conjugated.
\end{definition}

\begin{proposition}
The circles in $C_\Gamma$ corresponding to the edges 
of $P$ not lying in $\Gamma$ are linked if and only if the edges are conjugated.
\end{proposition}
\begin{proof}
This becomes evident if we look at the link in a way shown in Fig.~\ref{Octfig} on the right.
\end{proof}

\begin{lemma}\label{CEL}
Let $\Gamma$ be a Hamiltonian cycle in a simple $3$-polytope $P$. Then each edge in $M_\Gamma$ has a conjugated edge.
\end{lemma}
\begin{proof}
Indeed, let the edge $E$ of $M_\Gamma$ have no conjugated edges. $E$ is the 
intersection of two facets $F_i$ and $F_j$ of $P$ lying in the closure of
the same connected component of $\partial P\setminus\Gamma$. Then both vertices of $E$ belong to the same facet $F_k$
lying in the closure of the other connected component.  In this case $E$ belongs to $F_k$, which is a contradiction. 
\end{proof}
\begin{corollary}\label{cor:HClinked}
Each circle of the link $C_\Gamma$ corresponding to a Hamiltonian cycle $\Gamma$ 
in a simple $3$-polytope $P$ is linked to at least one 
other circle of $C_\Gamma$.
\end{corollary}
\begin{construction}[Transformation of a nonselfcrossing Eulerian cycle along conjugated vertices]
Given two conjugated vertices $v$ and $w$ of a nonselfcrossing Eulerian cycle $\gamma$
of an ideal right-angled $3$-polytope $Q$ we can build a new nonselfcrossing Eulerian cycle in the following
way. In both vertices we change the pairs of successive edges of the cycle to complementary pairs (see Fig.~\ref{flip}). 
In Fig.~\ref{Trans} we show how the Hamiltonian cycle $\Gamma_\gamma$ is transformed under this operation 
(the new Hamiltonian cycle belongs to another polytope obtained from $P_\gamma$ by two flips).
\begin{figure}
\begin{center}
\includegraphics[width=\textwidth]{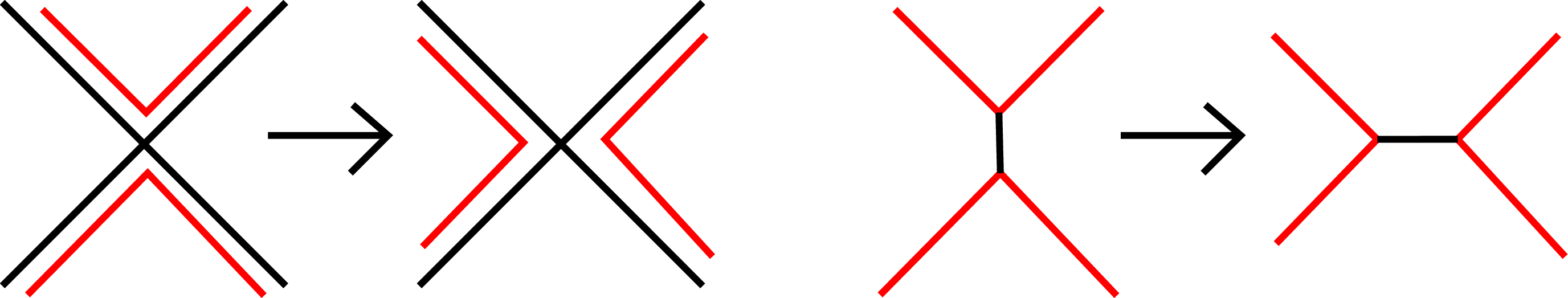}
\caption{Local transformation of the Eulerian cycle $\gamma$ and the corresponding flip of the polytope $P_\gamma$}
\label{flip}
\end{center}
\end{figure}
\begin{figure}
\begin{center}
\includegraphics[width=\textwidth]{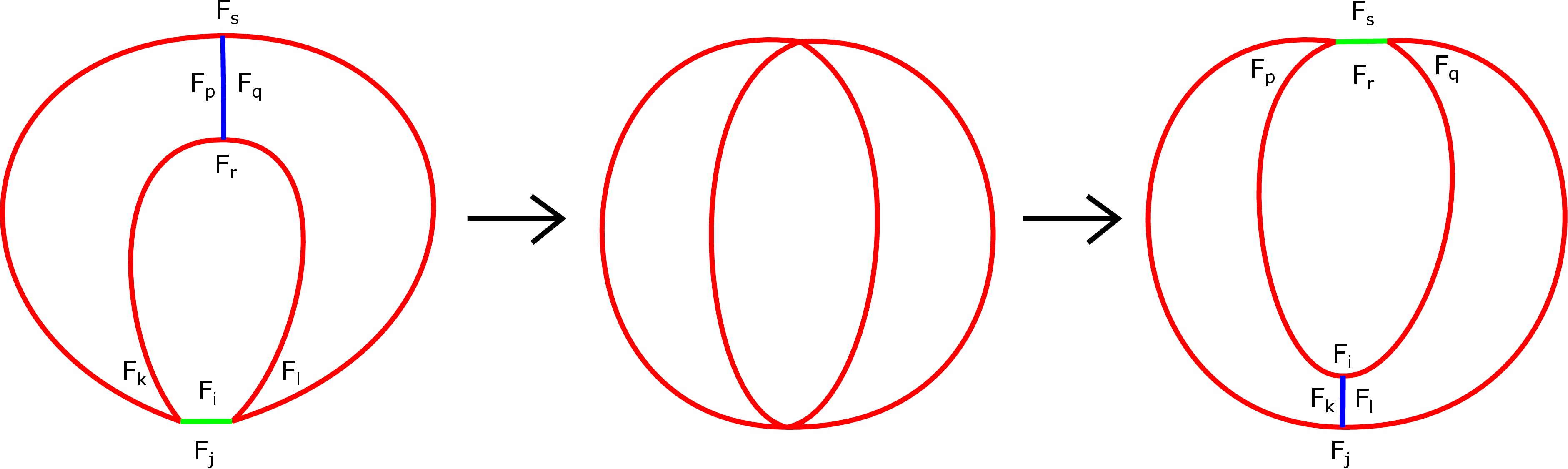}
\caption{Transformation of the Hamiltonian cycle $\Gamma_\gamma$.}
\label{Trans}
\end{center}
\end{figure}
In \cite{K68b} such a transformation of a Eulerian cycle is called a $\rho$-transformation. 
In \cite{FSW92} -- a $k_A$-transformation.
If the graph $G(Q)$ of $Q$ is a medial graph of the polytope $Q'$,
then $\gamma$ corresponds to a spanning tree $T$ in $G(Q')$ and $v$ and $w$ correspond to edges $E(v)$ and $E(w)$ such that exactly one of them belongs to $T$. 
The vertices are conjugated
if and only if the deletion of one of these edges from $T$ and the addition of the other edge produces a new spanning tree. It
corresponds to the transformed cycle.
\end{construction}
\begin{proposition}
Let $\gamma$ be a nonselfcrossing Eulerian cycle in the ideal right-angled $3$-polytope $Q$.
Then for any vertex of $Q$ there is at least one conjugated vertex and the corresponding transformation of $\gamma$.		
\end{proposition}
\begin{proof}
The proof follows from Lemma~\ref{CEL}.
\end{proof}
\cite[Theorem 17]{K68b}, \cite[Corollary 2]{FSW92}, \cite[Assertion before Section 3]{I11} and \cite[Theorem 2.22]{IM09}) 
imply the following result.
\begin{theorem}
Any two nonselfcrossing Eulerian cycles are connected by a sequence of transformations along conjugated vertices.
\end{theorem}

%Let $\gamma$ be a nonselfcrossing Eulerian cycle in a $4$-valent plane graph $G$. As in Construction \ref{con:EC}
%one can build a plane cubic graph $\widetilde{G}(G,\gamma)$. If $G$ is the graph of a polytope $P$, then 
%$\widetilde{G}(G,\gamma)$ if the graph of the simple $3$-polytope $Q(P,\gamma)$.
\begin{proposition}
Let $\gamma$ be a nonselfcrossing Eulerian cycle in a $4$-valent plane graph $G$. If $G_\gamma$ if the graph of a 
simple $3$-polytope $P_\gamma$, then for any other nonselfcrossing Eulerian cycle $\gamma'$ the graph
$G_{\gamma'}$ is also the graph of a simple $3$-polytope.
\end{proposition}
\begin{proof}
If $P_\gamma=\Delta^3$, then $G$ is a graph with $2$ vertices and $4$ edges, and the statement is clear.
Let $P_\gamma\ne\Delta^3$.
For a single transformation along conjugated vertices the corresponding plane graph $G_{\gamma'}$
is obtained from $G_\gamma$ by two flips of edges: at each edge we contract this edge and then divide the 
$4$-valent vertex into two vertices connected by an edge in the other manner. For a simple $3$-polytope $P$ and
its edge $E=F_i\cap F_j$ with vertices $F_i\cap F_j\cap F_k$ and $F_i\cap F_j\cap F_l$ the flip gives a new 
simple $3$-polytope if and only if $F_k\cap F_l=\varnothing$. This can be proved directly using the Steinitz theorem or it follows
from \cite[Theorem 11.4]{E19}: we first cut off the edge and then straighten along the perpendicular edge of the arisen quadrangle.
In our case the edges $E_1=F_i\cap F_j$ and $E_2=F_p\cap F_q$ of $P$ are conjugated along $\Gamma_\gamma$.
Let $F_i\cap F_j\cap F_k$ and $F_i\cap F_j\cap F_l$ be the 
vertices of $E_1$ and $F_p\cap F_q\cap F_r$ and $F_p\cap F_q\cap F_t$
be the vertices of $E_2$ (as it is denoted in Fig.~\ref{Trans}). Let $C_1$ and $C_2$  be two connected components of the complement 
$\partial P_\gamma\setminus\Gamma_\gamma$.  Both facets $F_k$ and $F_l$ lie in the closure of the same
component $C_i$ and are separated by $E_2$ in this component. Similarly, $F_r$ and $F_s$ lie in the closure of the 
other component and are separated by $E_1$ in this component. If $F_k\cap F_l=\varnothing$, then we can perform a flip
and obtain the new polytope $P'$. In this polytope the Hamiltonian cycle $\Gamma_\gamma$ is transformed into two disjoint
cycles containing all the vertices of $P'$. The complement to these cycles consists of two disks and a cylinder. Moreover, $F_r$ 
and $F_s$ in the polytope $P'$ lie in different disks. Therefore, they are disjoint and we can perform a flip at the edge
$E_2$ of $P'$ giving the polytope $P''=P_{\gamma'}$ we need. If $F_r\cap F_s=\varnothing$ in $P_\gamma$, then we can firs perform 
the flip at $E_2$, and then at $E_1$. If $F_k\cap F_l\ne\varnothing$,
then  $F_k\cap F_l=E_2$ and $F_k=F_p$, $F_l=F_q$. If also $F_r\cap F_s\ne\varnothing$, then
$F_r\cap F_s=E_1$ and $F_r=F_i$, $F_s=F_j$. Then any three of four facets $F_i$, $F_j$, $F_k$, and  $F_l$
intersect at a vertex and $P_\gamma=\Delta^3$, which is a contradiction. 

Since any two Eulerian cycles $\gamma$ and $\gamma'$ can be connected by a sequence of transformations
along conjugated vertices, the proposition is proved. 
\end{proof}

\begin{corollary}\label{corch4}
The class of $4$-valent plane graphs obtained by contraction of perfect matchings complementary to Hamiltonian cycles 
in simple $3$-polytopes consists of plane graphs without loops such that any face is bounded by a simple edge-cycle and for any pre-selected nonselfcrossing Eulerian cycle $\gamma$ the following condition holds: in the graph $G_\gamma$ if the boundary cycles of two faces intersect, then their intersection consists of exactly one edge.
\end{corollary}
\begin{remark}
The condition on loops is important, since the graph $G$ consisting of one vertex and two loops does not correspond to a simple $3$-polytope.  The graph $G_\gamma$ is a theta-subgraph.
\end{remark}
\begin{proof}[Proof of Corollary \ref{corch4}]
%Indeed, by the Steinitz theorem a simple  (that is, without loops and multiple edges) plane graph  with more than one vertex
%is combinatorially equivalent to a graph of a convex $3$-polytope if and only if
%any its face is bounded by a simple edge-cycle and if two boundary cycles intersect, then their intersection is a vertex or an edge.
We will use the Steinitz theorem~\ref{StT}.

Let $P$ be a simple $3$-polytope with a Hamiltonian cycle $\Gamma$. Consider its face $F_i$. Any edge of the perfect matching
$M_\Gamma$ complementary to $\Gamma$ either belongs to $F_i$, or intersects $F_i$ by one vertex, 
or does not intersect $F_i$. Hence, when we contract the edges of $M_\Gamma$, in the new graph it is again bounded by a 
simple edge-cycle. 

On the other hand, if any face of $G$ is bounded by a simple edge-cycle, then this
condition also holds in $G_\gamma$. Let us prove that the graph $G_\gamma$ is simple. 
It contains no loops, since $G$ contains no loops. If it contains two edges connecting the same vertices, consider one of these
vertices. Since $G_\gamma$ is $3$-valent, these two edges lie in the same face, and this face is a bigon. If the 
Hamiltonian cycle $\Gamma_\gamma$ contains both edges of this bigon, then $\Gamma_\gamma$ has no other edges, 
$G$ consists of one vertex  and two loops, which is a contradiction. If $\Gamma_\gamma$ does not contain one of the edges 
of the 
bigon, then in $G$ this bigon corresponds to a loop, which is also a contradiction. Thus, $G_\gamma$ is a simple graph. If also any nonempty intersection of two boundary cycles of faces of this
graph is an edge, then by the Steinitz theorem it is a graph of a simple $3$-polytope.
\end{proof}
\subsection{Links $C_\Gamma$ corresponding to Hamiltonian cycles in right-angled $3$-polytopes}
Theorem~\ref{thHL} (as well as  \cite[Theorems 6.5 and 11.6]{E19}) implies the following result.
\begin{proposition}\label{RAprop}
Let $\Gamma$ be a Hamiltonian cycle in a Pogorelov polytope $P$. Then $C_\Gamma$ is hyperbolic, and its
$2$-sheeted branched covering $N(P,\widetilde{\Lambda}_\Gamma)$ is a compact hyperbolic manifold.
\end{proposition}

\begin{example}
The dodecahedron is a unique Pogorelov polytope with minimal number of facets (equal to $12$).
Up to combinatorial symmetries it has a unique Hamiltonian cycle. In Fig.~\ref{HamD}, we show this Hamiltonian cycle, 
and the way how the associated ideal right-angled polytope is obtained from $A(4)$ by a sequence of two restricted edge-twists.
We also show another polytope corresponding to another Eulerian cycle in the $4$-antiprism. The corresponding links have homeomorphic complements, but the first link has the $2$-sheeted branched covering space with a hyperbolic structure, 
and the $2$-sheeted branched covering space of the second link contains incompressible tori corresponding
to $4$-belts  (see more details in \cite{E22a}).
\begin{figure}
\begin{center}
\includegraphics[width=\textwidth]{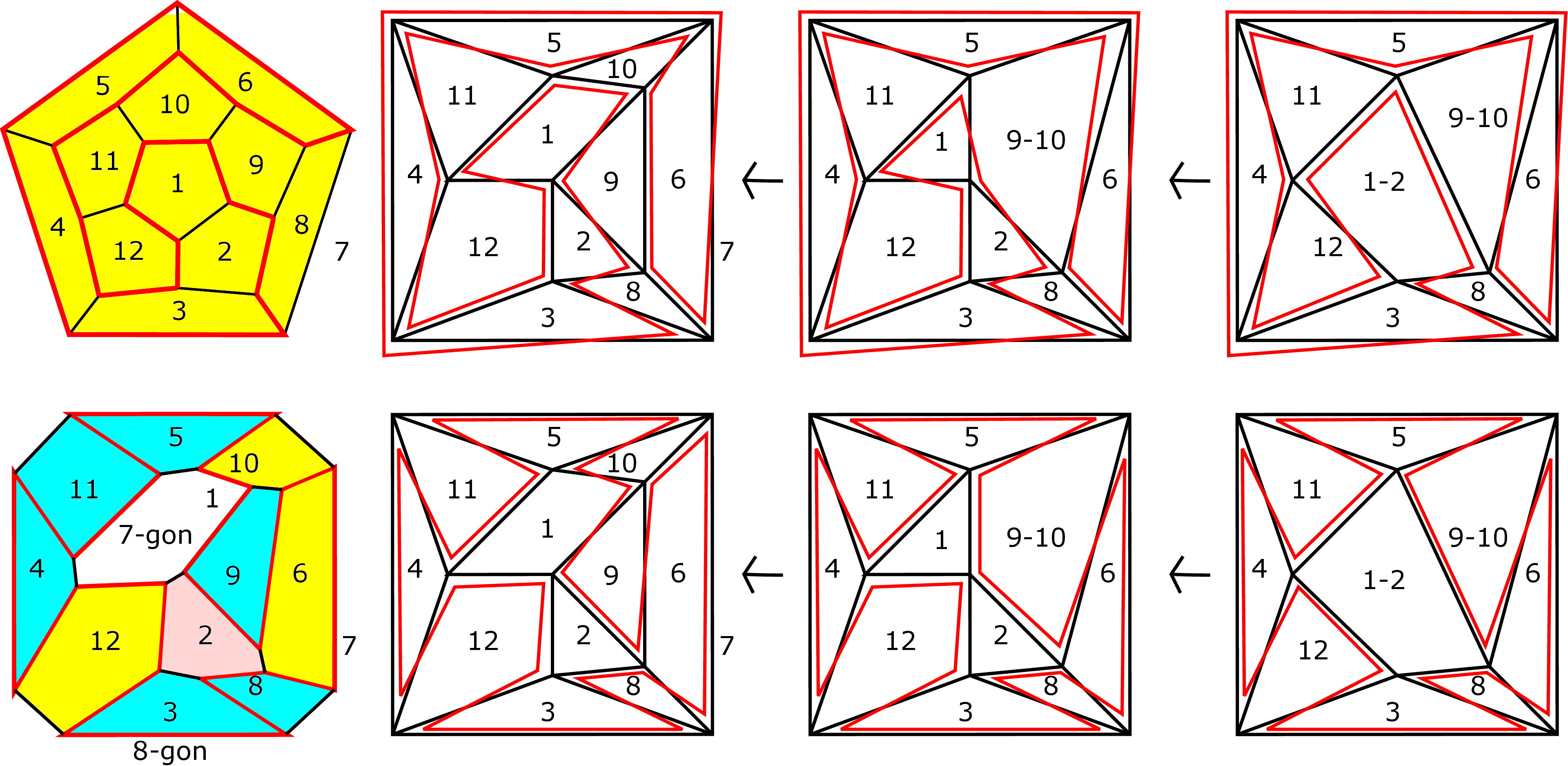}
\caption{Hamiltonian cycle in the dodecahedron}
\label{HamD}
\end{center}
\end{figure}
\end{example}
\begin{example}
A {\it fullerene} is a simple $3$-polytope with only pentagonal and hexagonal faces. It is known that any fullerene is a right-angled
hyperbolic polytope and the dodecahedron is the fullerene with minimal number of facets (see more details in \cite{E19}). 
It was shown by F.~Kardo\v{s} \cite{K20}
that any fullerene has a Hamiltonian cycle. Each Hamiltonian cycle in a fullerene corresponds to a hyperbolic link with hyperbolic $2$-sheeted branched covering space.
\end{example}

Let us denote by~$P_8$ the simple polytope  with $8$ facets drawn in the center at the bottom in Fig.~\ref{Octfig}.  It has a 
nontrivial $4$-belt consisting of pentagons and surrounding two quadrangles on each side. It was shown in \cite{E19} that the 
polytope $P_8$ has some properties similar to properties of almost Pogorelov polytopes. 
\begin{remark}
The polytope $P_8$ is known as the $3$-dimensional {\it pellytope} ($n$-dimensinal pellytope is a particular 
simple polytope with $3n-1$ facets. The number of its vertices is equal to {\it Pell's} number $N_{n+1}$, where
$N_1=1$, $N_2=2$, and $N_n=2N_{n-1}+N_{n-2}$.  Pellytopes determine a family of binary geometries, see \cite{BTTM25}. 
Such geometries have  been recently introduced in particle physics in connection with stringy integrals \cite{AHLT23}).
The polytope $P_8$ was the first example of a simple $3$-polytope whose moment-angle manifold 
$\mathcal{Z}_{P_8}$ has a nontrivial Massey product \cite{B03}.  Also $P_8$ is a unique {\it medial polytope} 
with $8$ facets (a simple $3$-polytope with all  facets $q$- and $(q+1)$-gons for some $q$)  and is a candidate to have the combinatorial type 
of a polytope with maximal volume among all $3$-polytopes with $8$ facets and given surface area \cite{G35}.  
\end{remark}
\begin{proposition}\label{APHC}
Let $\Gamma$ be a Hamiltonian cycle in an almost Pogorelov $3$-polytope $P$ or the polytope $P_8$. 
\begin{enumerate}
\item The link $C_\Gamma$ is hyperbolic if and only if each quadrangle of $P$ has three edges in~$\Gamma$.
\item The cube $I^3$ and the $5$-prism $M_5\times I$ do not have Hamiltonian cycles with the above condition. 
The polytope $P_8$ up to combinatorial symmetries has a unique Hamiltonian cycle with the above condition shown in Fig.~\ref{Octfig}.
\item For $P\notin \{I^3, M_5\times I, P_8\}$ the $2$-sheeted branched covering space corresponding to $C_\Gamma$ becomes hyperbolic after cutting along incompressible Klein bottles corresponding to quadrangles~of~$P$.  For $P=P_8$ the $2$-sheeted branched covering space splits into two manifolds with geometry $\mathbb L^2\times \mathbb R$ after cutting along the incompressible torus corresponding to the $4$-belt consisting of pentagons. 
\end{enumerate}
\end{proposition}
\begin{proof}
Item (1) follows from Theorem~\ref{thHL} (as well as  \cite[Theorems 6.5 and 11.6]{E19}), as well as 
from \cite[Corollary 12.31]{E19}. 

Item (2) follows from the direct enumeration of Hamiltonian cycles on $I^3, M_5\times I$ and $P_8$.

By \cite[Theorem 4.12]{E22a} for $P\notin \{I^3, M_5\times I, P_8\}$ its quadrangles correspond to incompressible Klein bottles in 
$N(P,\widetilde{\Lambda}_{\Gamma})$ such that the complement to their union has a complete hyperbolic structure of finite
volume. Also by this Theorem for $P=P_8$ the $4$-belt consisting of pentagons corresponds to an incompressible torus in
$N(P,\widetilde{\Lambda}_{\Gamma})$ such that its complement consists of two manifolds with geometry 
$\mathbb L^2\times \mathbb R$. This proves item (3).
\end{proof}
\begin{corollary}
A Hamiltonian cycle in a right-angled hyperbolic $3$-polytope of finite volume corresponds to a 
Hamiltonian cycle in the corresponding almost Pogorelov polytope 
if and only if at each ideal vertex it turns left or right, but does not go straight.
\end{corollary}
\begin{example}
In Fig.~\ref{PeH} we show a Hamiltonian cycle in the $3$-dimensional permutohedron intersecting each quadrangle by $3$ edges. 
After shrinking quadrangles to points we obtain a~Hamiltonian cycle in the ideal octahedron. 
\begin{figure}
\begin{center}
\includegraphics[width=.5\textwidth]{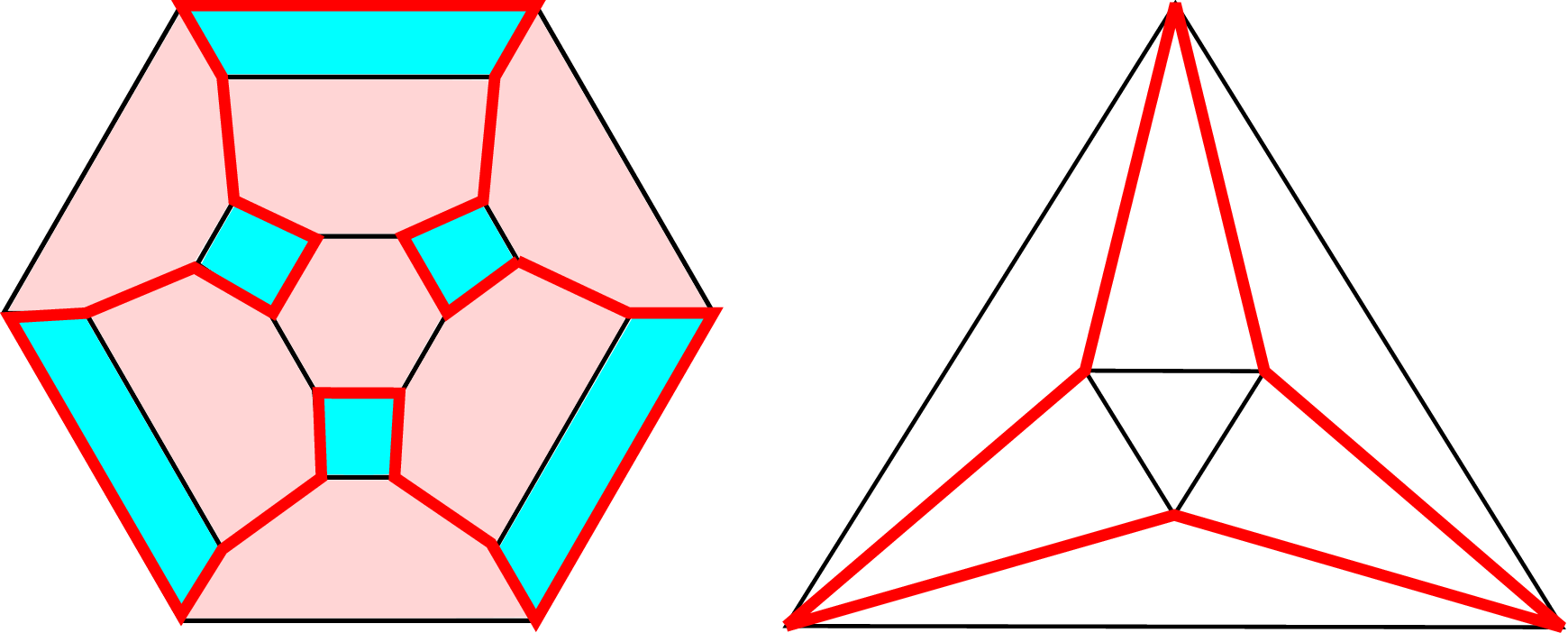}
\caption{The Hamiltonian cycle in the permutohedron corresponding to the~Hamiltonian cycle in the ideal octahedron.}
\label{PeH}
\end{center}
\end{figure}
\end{example}
\begin{question}
To characterise ideal right-angled $3$-polytopes corresponding to Hamiltonian cycles in (a) compact right-angled 
hyperbolic $3$-polytopes
(b) right-angled hyperbolic $3$-polytopes of finite volume. 
\end{question}
\begin{remark}
In \cite[Theorem 9.17]{E19}, is was proved that any ideal right-angled hyperbolic $3$-polytope $P$ 
can be obtained from an almost Pogorelov polytope or the polytope $P_8$  by a contraction of edges of a perfect matching such 
that no quadrangle contains two edges of the matching. Nevertheless, the complement to this matching may be not a Hamiltonian cycle, but a union of cycles containing all the vertices of the polytope.
\end{remark}

\section{Links $C_\Gamma$ consisting of mutually unlinked circles}\label{sec:TK4}
In this section we will give a criterion when the link $C_\Gamma$ corresponding to a Hamiltonian cycle, theta-subgraph or $K_4$-subgraph in a simple $3$-polytope $P$ consists of mutually unlinked circles. It is easy to see that this link consists of trivially
embedded circles.  
Moreover, as it was shown in Corollary~\ref{cor:HClinked} in the case of a Hamiltonian cycle in a simple $3$-polytope
each circle is linked to at least one other circle.

\begin{construction}[Cutting off a vertex of $\Gamma$]
Let $\Gamma$ be a~Hamiltonian theta-subgraph or a~Hamiltonian $K_4$-subgraph 
in a simple $3$-polytope $P$ and $v$ be one of its vertices. Then there is an operation of cutting off 
the vertex $v$, see Fig.~\ref{VcutH}. It produces a new polytope $P'$  with a triangle instead of the vertex $v$. If we chose one of the
three faces of $P$ (or, equivalently, $\Gamma$) containing $v$, then we can build uniquely a new Hamiltonian theta-subgraph
or $K_4$-subgraph $\Gamma'$ on $P'$ such that the edge of the new triangle corresponding to the chosen face 
belongs to $M_{\Gamma'}$. From the representation of $\Gamma$ on the coordinate rays of the octant 
with $v$ corresponding to the origin it is clear that the link  $C_{\Gamma'}$  is obtained from $C_\Gamma$
by an~addition of a trivial circle (for the theta-subgraph) or two trivial circles (for the $K_4$-subgraph) lying in
disjoint topological balls disjoint from  $C_\Gamma$. It follows from the~Steinitz theorem that for $P'\ne \Delta^3$ this operation is reversible: if
$P'$ has a triangle incident to a vertex $v$ of $\Gamma'$, then this triangle can be shrinked to obtain
a~new simple $3$-polytope $P$ with the Hamiltonian graph $\Gamma$ such that $(P',\Gamma')$
is obtained from $(P,\Gamma)$ but cutting off a vertex. On the level of graphs these operations correspond to the 
addition and the deletion of an edge near the vertex $v$.
\begin{figure}[h]
\begin{center}
\includegraphics[width=0.5\textwidth]{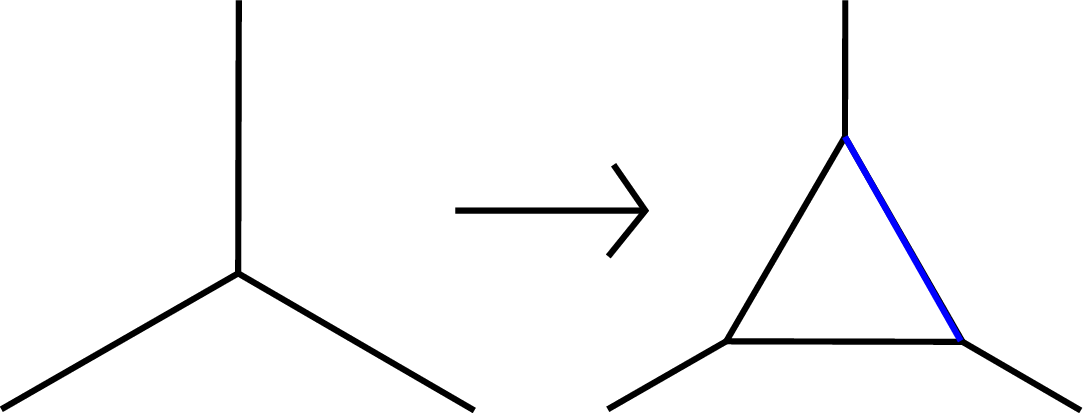}
\end{center}
\caption{Cutting off a vertex}\label{VcutH}
\end{figure}
\end{construction}
\begin{example}\label{ex:D30}
Let $\Gamma_0$ be the Hamiltonian theta-subgraph in the simplex $\Delta^3$, obtained by deletion of any edge 
from the graph $G(\Delta^3)=K_4$. Up to combinatorial symmetries it is a unique theta-subgraph in $G(\Delta^3)$.
The link $C_{\Gamma_0}$ is a trivial circle. Then for any pair $(P,\Gamma)$ obtained
from  $(\Delta^3,\Gamma_0)$ by a sequence of operations of cutting off a vertex the corresponding link~$C_{\Gamma}$
is trivial.
\end{example}
\begin{theorem}\label{Th:unl}
Let $\Gamma$ be a Hamiltonian theta-subgraph in a simple $3$-polytope $P$. Then 
\begin{enumerate}
\item the link $C_\Gamma$ consists of mutually unlinked circles if and only if each edge of $M_{\Gamma}$
connects vertices of different paths of $\Gamma$;
\item if $C_\Gamma$ consists of mutually unlinked circles and is nontrivial, then it contains 
a triple of Borromean rings; 
\item the link $C_\Gamma$ is trivial if and only if  $(P,\Gamma)$ is obtained from $(\Delta^3,\Gamma_0)$,
by a sequence of operations of cutting off a vertex.
\end{enumerate}
\end{theorem}
\begin{proof}[Proof of Theorem~\ref{Th:unl}]
From the representation of the theta-subgraph on the coordinate rays of the octant 
it is clear that each two circles are unlinked if each edge of the matching connects two different paths of $\Gamma$.
On the other hand, if there is an edge $M_\Gamma$
connecting two vertices $v$ and $w$ on the same path, then take such an edge $E_1$ 
with the condition that between $v$ and $w$ there are no pairs of vertices connected by edges in $M_\Gamma$. 
There is a vertex of another edge $E_2$ lying on the same path between $v$ and $w$, 
for otherwise there is a bigonal face, which is a contradiction. The edge $E_2$ lies 
in another connected component of $\partial P\setminus\Gamma$. The other vertex of $E_2$ lies
either on the same path, or on another path. In both cases it is clear from the octant representation
that the circles are linked is the standard way (as in the Hopf link). This proves~(1).

Now let $C_\Gamma$ consist of mutually unlinked circles. By (1) each edge of $M_{\Gamma}$ connects vertices on different paths
of $\Gamma$. Let $\Gamma_1$, $\Gamma_2$ and $\Gamma_3$ be the paths of $\Gamma$ connecting the vertices $v$ and $w$.
If there is an edge $E\in M_\Gamma$ with vertices $v_i$ and $v_j$ on $\Gamma_i$ and $\Gamma_j$ such that there are no vertices 
on these paths between $v$ and $v_i$ and $v$ and $v_j$,
then $P$ has a triangle incident to $v$ and if $P\ne \Delta^3$, then $(P,\Gamma)$ is obtained from 
some pair $(P',\Gamma')$ by cutting off a vertex.  The graph $G(P')$ is obtained from $G(P)$ by deletion of $E$. 
The link $C_{\Gamma}$ is trivial if and only if $C_{\Gamma'}$ is trivial. If $P$ has no such edges $E$,
then consider a vertex $v_1$ on $\Gamma_1$ closest to $v$. If this vertex is $w$, then all the edges of $M_{\Gamma}$ connect 
vertices on $\Gamma_2$ and $\Gamma_3$ and are ``parallel'', in particular the first and the last edges are of the above type.
A contradiction. Thus, $v_1$ belongs to some edge $E_1\in M_{\Gamma}$ with the other vertex $v_2$ 
lying on the other path, say $\Gamma_2$. By our assumption, there is a vertex $v_3$ between $v$ and $v_2$. Let $v_3$
be the closest vertex to $v$.  Then $v_3\in E_2\in M_\Gamma$. The other vertex $v_4$ of $E_2$ belongs to $\Gamma_3$.
Again by our assumption there is a vertex $v_5$ between $v$ and $v_4$, $v_5\in E_3\in M_\Gamma$. Let $v_6$
be the other vertex of $E_3$. Then $v_6\in \Gamma_1$ and $v_1$ lies between $v$ and $v_6$. The edges $E_1$,
$E_2$ and $E_3$ correspond to Borromean rings in $C_\Gamma$. In particular, $C_\Gamma$ is a nontrivial link.
Thus, if $C_\Gamma$ does not contain Borromean rings, then $(P,\Gamma)$ is obtained from 
$(\Delta^3,\Gamma_0)$ by a sequence of operations of cutting off a vertex. In particular, $C_\Gamma$ is trivial. Together
with Example~\ref{ex:D30} this proves (2) and (3).
\end{proof}
\begin{example}\label{exthetaD}
In Fig.~\ref{Dodeclink} we show the link $C_\Gamma$ 
corresponding to a Hamiltonian theta-subgraph $\Gamma$ in the dodecahedron $P$.
It consists of $9$ mutually unlinked circles and contains many triples of Borromean rings.
The manifold $N(P,\widetilde{\Lambda}_{\Gamma})$ has a hyperbolic structure. The polytope $P_\Gamma$ 
is a right-angled hyperbolic polytope of finite volume with $2$ proper and $9$ ideal vertices. 
The complement $S^3\setminus C_{\Gamma}$ is glued of $8$ copies of $P_\Gamma$.
\end{example}
%\begin{figure}
%\begin{center}
%\includegraphics[width=\textwidth]{Borromeo-bold.eps}
%\caption{The Borromean rings corresponding to a Hamiltonian theta-subgraph on the cube}
%\label{Borromeo}
%\end{center}
%\end{figure}
\begin{figure}
\begin{center}
\includegraphics[width=\textwidth]{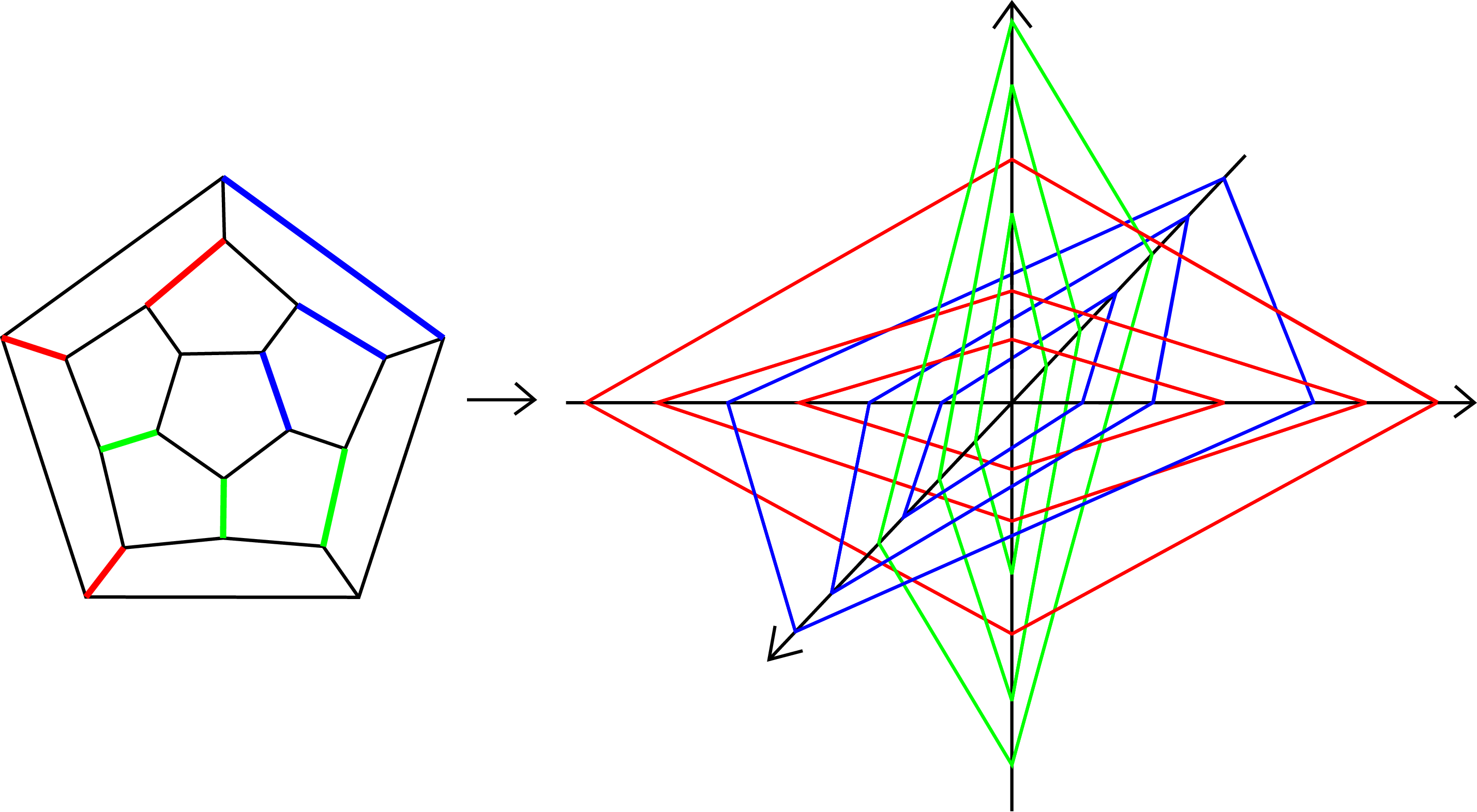}
\caption{The link corresponding to a Hamiltonian theta-subgraph in the dodecahedron}
\label{Dodeclink}
\end{center}
\end{figure}

\begin{definition}
For a segment $I$ and a point $x\notin I$ denote by $x*I$ the triangle spanned by $v$ and $I$.
\end{definition}
\begin{theorem}\label{K4UL}
Let $\Gamma$ be a Hamiltonian $K_4$-subgraph in a simple $3$-polytope $P$.
Then 
\begin{enumerate}
\item the link $C_\Gamma$ consists of mutually unlinked circles if and only if $M_{\Gamma}$ splits
into matchings $M_{\Gamma}(v_i)$ corresponding to vertices of $K_4$, such that each matching 
consists of edges connecting the vertices on different paths of $K_4$ containing $v_i$
and for any two edges $E_1\in M_{\Gamma}(v_i)$ and $E_2\in M_{\Gamma}(v_j)$, $i\ne j$
the triangles $v_i*E_1$ and $v_j*E_2$ do not intersect;
\item if $C_\Gamma$ consists of mutually unlinked circles and is nontrivial, then it contains 
a triple of Borromean rings; 
\item the link $C_\Gamma$ is trivial if and only if  $(P,\Gamma)$ is obtained from $(\Delta^3,G(\Delta^3))$
by a sequence of operations of cutting off a vertex.
\end{enumerate}
\end{theorem}
\begin{proof}
If the condition of item (1) holds, then we can isotope all the matchings to be close to the corresponding vertices.
Then near each vertex we have the theta-subgraph $\Gamma(v_i)$ (obtained by shrinking to point the triangle of $\Gamma$
complementary to $v_i$) with the matching $M_{\Gamma}(v_i)$, and the link $C_\Gamma$ 
consists of two copies of each link $C_\Gamma(v_i)$ for all $i$ 
lying in disjoint disks. So $C_\Gamma$ consists of mutually unlinked circles.

Now let $C_\Gamma$ consist of mutually unlinked circles. If there is an edge in $M_\Gamma$ connecting the vertices
on the same path of $\Gamma$, then consider such an edge $E_2$ with ends $w_1$ and $w_2$ and no other pairs of vertices
between $w_1$ and $w_2$ connected by an edge in $M_\Gamma$. Since $P$ has no bigons, there is a vertex $w_3$ between $w_1$ and $w_2$. This vertex is connected by an edge $E_2\in M_\Gamma$ to some other vertex $w_4$. It is clear from 
the representation of $K_4$ as the graph of the simplex $\Delta^3$ 
with vertices $v_0=(0,0,0)$, $v_1=(1,0,0)$, $v_2=(0,1,0)$, $v_3=(0,0,1)$ that either
$E_1$ and $E_2$ correspond to $4$ pairs of circles linked in a standard way (if $w_4$ lies in the same path), or
$E_1$ corresponds to $4$ and $E_2$ corresponds to $2$ unlinked circles such that each circle of the second type
is linked in a standard way to two circles of the first type. A contradiction. Thus, each edge of  $M_\Gamma$
connects two vertices on different paths. Then for each edge $E\in M_\Gamma$ there is a unique vertex $v_i$ of $K_4$
such that $E$ lies in the triangle $[v_i,v_j,v_k]$ and connects the points on the paths $[v_i,v_j]$ and $[v_i,v_k]$.
Denote this vertex $v(E)$. Now the proof of item (1) follows from
\begin{lemma}
The link corresponding to two edges $E_1,E_2\in M_\Gamma$ connecting vertices on different paths of $K_4$
is nontrivial if and only if $v(E_1)\ne v(E_2)$ and the triangles $v(E_1)*E_1$ and $v(E_2)*E_2$ intersect 
(equivalently, the segments between a vertex of $E_1$ and $v(E_1)$ and a vertex of $E_2$ and $v(E_2)$ lying
both on $[v(E_1), v(E_2)]$ intersect). If this link is nontrivial, then it is the $4$-link chain like in Example~\ref{ex:ak}. 
\end{lemma}
\begin{proof} 
If $v(E_1)=v(E_2)$, then as in the above argument $C_\Gamma$ consists of two copies of the trivial link corresponding to
two edges on the theta-subgraph. These copies lie in disjoint disks, so $C_\Gamma$ is trivial.

If $v(E_1)\ne v(E_2)$ and  $v(E_1)*E_1\cap v(E_2)*E_2=\varnothing$, then $C_\Gamma$ is trivial, since it 
consists of four circles lying in disjoint balls.

If $v(E_1)\ne v(E_2)$ and  $v(E_1)*E_1\cap v(E_2)*E_2\ne\varnothing$, then the edge $[v(E_1),v(E_2)]$
contains the vertex $w_1$ of $E_1$ and the vertex $w_2$ of $E_2$, and these vertices lie in the order $(v(E_1), w_2, w_1, v(E_2))$.
Then each of the circles corresponding to $E_1$ is linked to each of the circles corresponding to $E_2$
in a standard way (like in the Hopf link) and these circles form the $4$-link chain. This finishes the proof.
\end{proof}
The proof of items (2) and (3) is the same as the proof of items (2) and (3) of Theorem~\ref{Th:unl}. 
\end{proof}

\section{Higher-dimensional hyperbolic links corresponding to generalized $VM$-hyperelliptic manifolds}
Here we consider ($n-2$)-dimensional generalizations of the links 
$C_\tau\subset N(P,\Lambda_\tau)$ and $C_\Gamma\subset S^3$. 

Let $\Lambda$ be a linearly independent vector-coloring of rank $r$ of a simple $n$-polytope $P$. Assume that
$\tau\in \mathbb Z_2^r$ is an involution such that $N(P,\Lambda)/\langle\tau\rangle$ is a topological manifold.
By \cite[Proposition 4.15]{E26} this holds if and only if for any vertex $v=F_{i_1}\cap\dots\cap F_{i_n}$
either $\tau\notin\langle \Lambda_{i_1},\dots,\Lambda_{i_n}\rangle$, or $\tau=\Lambda_{i_k}+\Lambda_{i_l}$
for $k,l\in\{i_1,\dots, i_n\}$.

Let $M_1$, $\dots$, $M_s$ be the set of all $(n-2)$-dimensional faces  $F_i\cap F_j$ such that  
$\Lambda_i+\Lambda_j=\tau$. Then $M_a\cap M_b=\varnothing$ for $a\ne b$, for otherwise $\Lambda$ is not linearly
independent. Set $M_\tau=M_1\sqcup\dots\sqcup M_s$. This is an $(n-2)$-dimensional generalization of the matching $M_\tau$ 

According to \cite[Proposition 4.31]{E26} the projection mapping 
$N(P,\Lambda)\to N(P,\Lambda)/\langle\tau\rangle=N(P,\Lambda_\tau)$ 
is a $2$-sheeted branched covering with the branch set $C_\tau$ 
(homeomorphic to the set of fixed points of $\tau$ in $N(P,\Lambda)$) being the disjoint union of copies of the closed 
$(n-2)$-manifolds $N(M_q,\Lambda_{M_q})$, where for $M_q=F_i\cap F_j$ the vector-coloring $\Lambda_{M_q}$
maps each facet $F_i\cap F_j\cap F_k$ of $M_q$ to $[\Lambda_k]\in\mathbb Z_2^r/\langle\Lambda_i,\Lambda_j\rangle$.
If $r_q$ is rank of $\Lambda_{M_q}$, then $N(M_q,\Lambda_{M_q})$ is represented in $C_\tau$ by  $2^{r-r_q-2}$ copies.
We will call the manifold $C_\tau\subset N(P,\Lambda_\tau)$ a {\it link}.

There arises a natural question.
\begin{problem}
To find a criterion when the complement $N(P,\Lambda_\tau)\setminus C_\tau$
has a complete hyperbolic structure of finite volume.
\end{problem}
Theorem~\ref{thHL} gives a complete answer to this problem in dimension $n=3$. For general $n$ we will give a sufficient condition. Consider the cell complex $P/M$ obtained by shrinking to points $[M_i]$ all the faces $M_i$ of $P$.
\begin{theorem}
If there is a combinatorial equivalence $P/M\simeq Q$, where $Q$ is a right-angled polytope 
of finite volume in $\mathbb L^n$ and the points $[M_q]$ correspond to ideal vertices, then the complement to $C_\tau$ in the manifold 
$N(P,\Lambda_\tau)$ is homeomorphic to $N(Q,\Lambda_Q)$, where $\Lambda_Q$ is a linearly independent vector-coloring
of rank $r$ of $Q$ induced from $\Lambda_\tau$  by the combinatorial equivalence.  In particular, this complement has a complete hyperbolic structure of finite volume.  
\end{theorem}
\begin{proof}
Indeed, if there is a combinatorial equivalence $P/M\simeq Q$, then there is a homeomorphism
$P\setminus M\to Q\subset \mathbb L^n$. This homeomorphism can be constructed as follows.
If we cut off each face $M_i$ by a hyperplane parallel to a supporting hyperplane and close to it, 
then the resulting simple polytope $P'$ 
is combinatorially equivalent to the simple polytope $Q'$ obtained from $Q$ by cutting off each vertex.
Hence, $P'$ and $Q'$ are homeomorphic via barycentric subdivisions.
\begin{lemma}
If there is a combinatorial equivalence $P/M\simeq  Q$, where $Q$ is a right-angled polytope 
of finite volume in $\mathbb L^n$, then each face $M_i$ is combinatorially equivalent to the cube~$I^{n-2}$.
\end{lemma}
\begin{proof} 
Near each face $M_i$ of $P$
the polytope $P'$ has the facet combinatorially equivalent to $M_i\times I$, and the part of $P$ between this facet
and $M_i$ is combinatorially equivalent to $M_i\times \Delta^2$ (see more details in \cite[Construction 4.32]{E26}).
Near each ideal vertex of $Q$ the polytope $Q'$ has the facet combinatorially equivalent to the cube $I^{n-1}$ 
(since at this vertex the intersection of a horosphere with $Q$ is a Euclidean right-angled polytope),
and the part of $Q$ between this facet and vertex is the pyramid $CI^{n-1}$. Thus, $M_i\times I\simeq I^{n-1}$, and
$M_i\simeq I^{n-2}$ (a decomposition of a simple polytope into the product of polytopes irreducible
with respect to the product operation is unique, see \cite[Proposition A.1]{BE11}).
\end{proof} 
Thus, each part of $P\setminus M$ outside $P'$ is homeomorphic via barycentric subdivisions to 
$(I^{n-2}\times \Delta^2)\setminus (I^{n-2}\times \{v\})$, 
and each part of $Q$ outside $Q'$ is homeomorphic to $(CI^{n-1})\setminus\{{\rm apex}\}$. These homeomorphisms are
compatible with the homeomorphism $P'\to Q'$. At last, there are homeomorphisms 
$(I^{n-2}\times \Delta^2)\setminus (I^{n-2}\times \{v\})\to I^{n-2}\times (I^2\setminus I^1)\simeq I^n\setminus I^{n-1}$ and 
$(CI^{n-1}\setminus\{{\rm apex}\})\to I^n\setminus I^{n-1}$ compatible on the facet $I^{n-1}$
corresponding to $M_i\times I$.

This homeomorphism induces the homeomorphism $N(P,\Lambda_\tau)\setminus C_\tau\simeq N(Q,\Lambda_Q)$. 
\end{proof}

For general $n$ there is the following construction of $VM$-hyperelliptic manifold from \cite{E26}: 

Let $M$ be a disjoint set of $(n-2)$-facets of $P$ such that if we glue the facets of $P$ along the faces in $M$ the resulting complex $\Gamma$ is equivalent to the polyhedron $S^n_{k,\geqslant 0}$ defined in the sphere 
$S^n=\{(x_1,\dots, x_{n+1})\in\mathbb R^{n+1}\colon x_1^2+\dots+x_{n+1}^2=1\}$ as the intersection 
of $k$ coordinate halfspaces $x_i\geqslant 0$, $i=1,\dots,k$. 
Such complex $\Gamma$ is called in \cite{E26} a~Hamiltonian $\mathcal{C}(n,k)$-subcomplex. 
It can be shown that $k\in\{n-1,n,n+1\}$. For $n=3$ any Hamiltonian
$C(3,2)$-subcomplex is defined by a Hamiltonian cycle, any $C(3,3)$-subcomplex -- by a
theta-subgraph, and any 
$C(3,4)$-subcomplex -- by a $K_4$-subgraph. Any Hamiltonian $C(n,k)$-subcomplex
in a simple $n$-polytope $P$ defines a linearly independent vector-coloring $\widetilde{\Lambda}_{\Gamma}$
of $P$ of rank $k+1$ such that $N(P,\widetilde{\Lambda}_{\Gamma})$ is a generalized $VM$-hyperelliptic manifold
with a canonical generalized $VM$-hyperelliptic involution $\tau$. 
The corresponding branch set $C_\tau$ we denote $C_\Gamma$. It was proved in \cite[Corollary 4.1]{E26} that in dimensions 
$n\leqslant 4$ any generalized $VM$-hyperelliptic manifold arises in this way.

\begin{corollary}\label{corcnk}
Let $k\in\{n-1,n,n+1\}$, $P$ be a simple $n$-polytope with $m$ facets,
and $M$ be a disjoint set of its  $m-k$ faces of dimension $(n-2)$ such that 
\begin{enumerate}
\item each face in $M$ is combinatorially equivalent to the cube $I^{n-2}$;
\item $M$ defines a Hamiltonian $\mathcal{C}(n,k)$-subcomplex $\Gamma$ in $P$;
\item The complex $P/M$  is combinatorially equivalent to a right-angled polytope $Q\subset\mathbb L^n$ of finite volume, where shrinked faces correspond to ideal vertices.
\end{enumerate}
Then the $(n-2)$-dimensional link $C_\Gamma\subset S^3$ has a complete hyperbolic structure of finite volume in its complement.
Moreover, this complement is partitioned into $2^k$ copies of $Q$. For $k=n-1$ any vertex of $P$ lies on $M$ and $Q$ has only ideal vertices. For $k=n$ all but $2$ vertices of $P$ lie on $M$ and $Q$ has two finite vertices except for ideal vertices. 
For $k=n+1$ all but $n+1$ vertices of $P$ lie on $M$ and $Q$ has $n+1$ finite vertices except for ideal vertices.
\end{corollary}

\begin{example}\label{ex:l2l2}
In Fig.~\ref{5x5} we show a set $M$ of five disjoint quadrangles  defining 
a~Hamiltonian $\mathcal{C}(4,5)$-subcomplex $\Gamma$ in the polytope $P$, which is the~product of two pentagons. 
The link $C_\Gamma\subset S^4$ is~a~disjoint union of~five tori $\mathbb T^2=S^1\times S^1$
known as Ivan\v{s}i\'{c} link \cite{I04,I12} (this follows from \cite{M24}, see also \cite[Example 4.76]{E26}). As follows from \cite{M24},  $P/M$ is combinatorially equivalent to a right-angled polytope $Q$ of finite volume introduced in \cite{RT00,PV05}. It has $5$ ideal and $5$ finite vertices. The dual polytope
is combinatorially equivalent to the hypersimplex $\Delta(5,2)$. The complement to $C_\Gamma$ in $S^4$ is glued of $32$
copies of $Q$.
\begin{figure}[h]
\begin{center}
\includegraphics[width=0.5\textwidth]{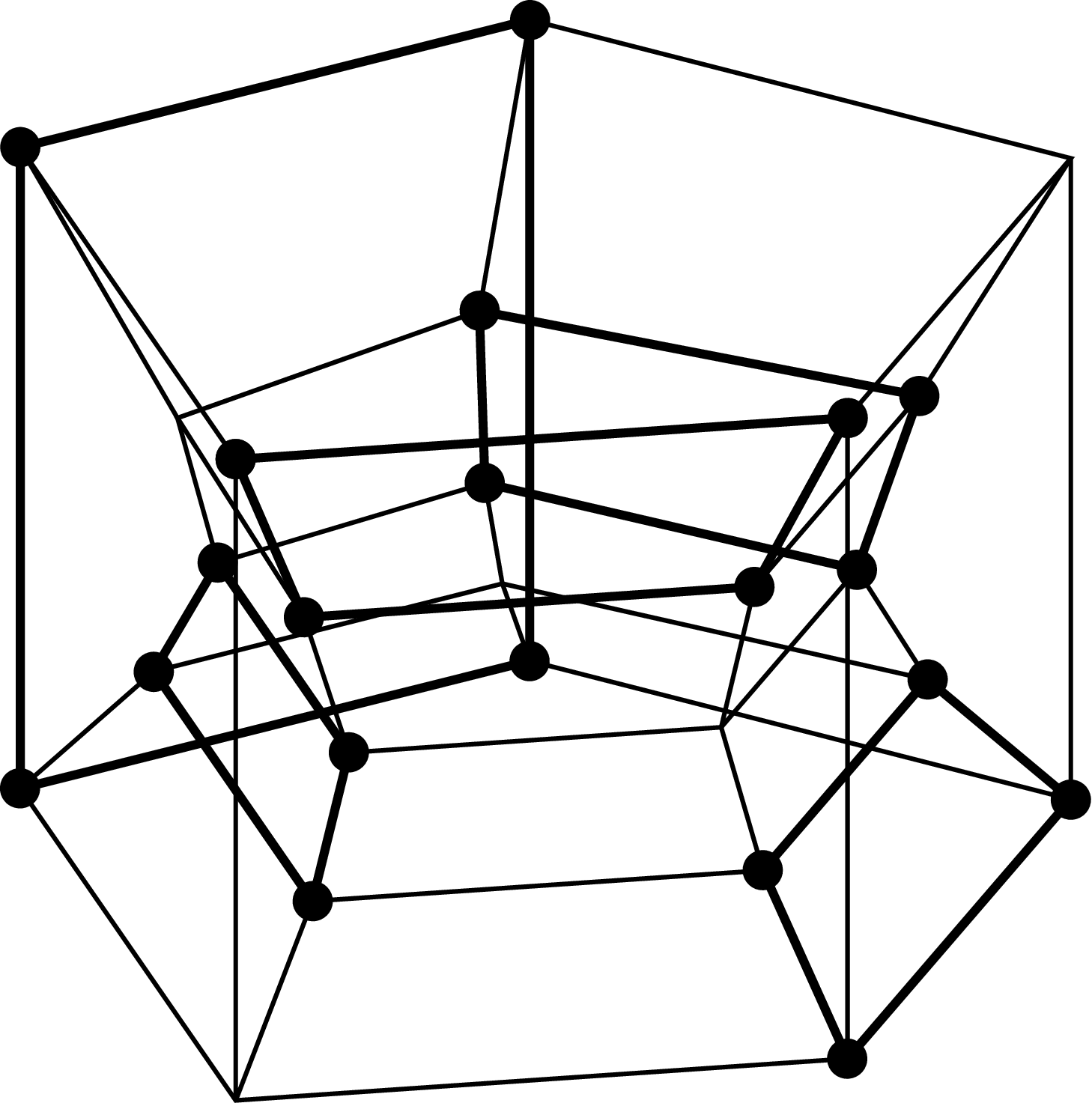}
\end{center}
\caption{A set of disjoint quadrangles defining a~Hamiltonian $\mathcal{C}(4,5)$-subcomplex 
in the product of two pentagons}\label{5x5}
\end{figure}
\end{example}
\begin{problem}
To find more examples for Corollary~\ref{corcnk} in dimensions $n\geqslant 4$.  
\end{problem}
\section{Acknowledgements}
The author is grateful to V.M.~Buchstaber for stimulating questions that led to the expansion of the scope of this article.
The author is grateful to N.V.~Bogachev, A.A.~Gaifullin, V.Yu.~Gorchakov,  D.V.~Gugnin, D.P.~Ilyutko, I.Yu.~Limonchenko,  A.D.~Mednykh, T.E.~Panov, M.V.~Prasolov, V.A.~Shastin,  D.V.~Talalaev, D.A.~Tsygankov, A.Yu.~Vesnin  for fruitful discussions.

\end{document}